\documentclass[dvipsnames]{amsart}
\usepackage[utf8]{inputenc}
\usepackage[margin=1.2in]{geometry}
\usepackage{xypic}

\newcommand{\Jcom}[1]{\textcolor{blue}{[J:#1]}}

\renewcommand{\phi}{\varphi}
\newcommand{\id}{\mathrm{id}}
\newcommand{\comment}[1]{}
\newcommand{\Xsing}{X^\tsing}
\newcommand{\ol}[1]{\overline{#1}}
\newcommand{\Wedge}{\bigwedge}
\newcommand{\im}{\operatorname{im}}
\newcommand{\FF}{\mathbb{F}}

\usepackage{macros}
\title{Singular loci in varieties of tensors}

\author[Chiu]{Christopher H.~Chiu}
\address{University of Bern, Mathematical Institute, Sidlerstrasse 5,
3012 Bern, Switzerland; and KU Leuven, Department of Mathematics, Celestijnenlaan 200b, Box 02400, 3001 Heverlee, Belgium.}
\email{christopher.chiu@unibe.ch}
\thanks{CC was supported by SNSF postdoc fellowship 217058 and FWO junior postdoc fellowship 12AZ524N.}

\author[Danelon]{Alessandro Danelon}
\address{University of Michigan, 530 Church St, Ann Arbor, Michigan 48109}
\email{adanelon@umich.edu}

\author[Draisma]{Jan Draisma}
\address{University of Bern, Mathematical Institute, Sidlerstrasse 5,
3012 Bern, Switzerland}
\email{jan.draisma@unibe.ch}
\thanks{JD was partially supported by project grant 200021-227864 from the Swiss National Science Foundation}

\begin{document}

\begin{abstract}
A $\bVec$-variety is a suitable functor $X$ from finite-dimensional vector
spaces to finite-dimensional varieties. Most varieties in the geometry
of tensors, e.g.\ the variety of
$d$-way tensors of slice rank at most $r$, are of this form. 
We prove that the singular locus
of a $\bVec$-variety is a proper closed $\bVec$-subvariety, analogously
to the situation for ordinary finite-dimensional varieties. Via earlier
work of the third author, this implies that these singular loci admit
a description by finitely many polynomial equations.

A natural follow-up question to our main result is whether a $\bVec$-variety
also admits a suitably functorial resolution of singularities. We
establish some preliminary results in this direction in the
regime where the dimension of $X(V)$ grows linearly with that of $V$.
\end{abstract}


\maketitle

\setcounter{tocdepth}{1}

\section{Introduction}\label{sec:Introduction}

\subsection{Main result}

A crucial observation in the geometry of tensors is that many varieties of
tensors depend functorially on the choice of a finite-dimensional vector
space. The goal of this paper is to show that the singular locus of such
a variety exhibits the same functoriality, and can hence be described in
finite terms. Our main result is a vast generalisation of the following
well-known example.

\begin{xmp}\label{xmp:singularlocusboundedmatrices}
The singular locus of the variety of $n \times n$-matrices of rank at most $r$
is the variety of matrices of rank at most $r-1$. This statement is
independent of $n$, provided that $n>r$. 
\end{xmp}

Let $K$ be a field of characteristic zero and let $\bVec$ denote the
category of finite-dimensional vector spaces over $K$. The varieties
that we consider are certain functors $X$ from $\bVec$ to the category of
affine schemes of finite type. For $V \in \bVec$, the variety $X(V)$ is
a closed subvariety of an affine space $P(V)$, where
$P$ is a direct sum of Schur functors. Thus $P(V)$ can be thought of as a
space of tensors of fixed type whose size varies with $V$, and $X(V)$ is
a closed subvariety, functorial in $V$; see Section~\ref{sec:Background}
for details.  

In the example above, $P(V)$ equals $V \otimes V$, and $X$ sends
$V$ to the variety of tensors in $P(V)$ of rank at most $r$. Note that for a
linear map $\phi:V \to W$, the linear map $P(\phi)=\phi \otimes \phi$
maps $X(V)$ into $X(W)$, so that $X$ is, indeed, a functor from
$\bVec$ to the category of varieties over $K$.

Given a $\bVec$-variety $X$, we look at the behaviour of $\oSing(X(V))$,
the singular locus of $X(V)$, as $V$ varies. The following is our main result.

\begin{thm}[Main Theorem]\label{thm:mainintro}
    Let $X$ be a $\bVec$-variety.  Then there exists a unique closed
    $\bVec$-subvariety $Y$ of $X$ such that $Y(V)=\oSing(X(V))$ for
    all $V \in \bVec$ with $\dim(V) \gg 0$. This $Y$ satisfies $Y(V)
    \supset \oSing(X(V))$ for all $V \in \bVec$. 
\end{thm}

We will denote this $\bVec$-subvariety $Y$ by $\Xsing$.  In the example
above, $\Xsing(V)$ is the variety of matrices in $V \otimes V$ of rank at
most $r-1$. Note that for all $V$ of dimension at most $r$, $X(V)$ equals
$P(V)$ and is hence smooth. This shows that the inclusion $\Xsing(V)
\supset \oSing(X(V))$ can indeed be strict for small $V$.

\subsection{Local equations for $X$ away from $\Xsing$}

Let $X$ be a closed $\bVec$-subvariety of a polynomial functor $P$. By
topological Noetherianity (see Theorem~\ref{thm:Noetherian} below)
there exists a $U \in \bVec$ such that for all $V \in \bVec$ we have
\[ X(V)=\{p \in P(V) \mid \forall \phi \in \oHom(V,U):
P(\phi)(p) \in X(U)\}. \]
In particular, all pullbacks of defining equations in $K[P(U)]$ for
$X(U)$, along the linear maps $P(\phi)$, together define $X(V)$ as a
set. However,
it is {\em not known} whether the corresponding statement holds at the
ideal-theoretical level: do those pullbacks generate the ideal of $X(V)$
for all $V$, if $U$ is taken sufficiently large? The following result
shows that at least outside singularities these pullbacks do define a
reduced variety.

\begin{thm} \label{thm:LocalEqs}
Let $X$ be a closed $\bVec$-subvariety of a polynomial functor $P$.
Then there exists a $U \in \bVec$ such that for all $V \in \bVec$ the ideal
\[
    I_V \coloneqq (f \circ P(\phi) \mid f \in I(X(U)), \phi \in \oHom(V,U) )
\]
of $K[P(V)]$ defines $X(V)$ as a set, and the reduced locus of $\oSpec(K[P(V)]/I_V)$ contains
$X(V)\setminus \oSing(X(V))$.
\end{thm}

In fact, by the Jacobian criterion the second statement is equivalent to the following. If $p \in X(V)$ is nonsingular and $I_V$ is generated by $f_1,\ldots,f_s$, then the evaluation of some $c \times c$-minor of the Jacobian matrix $(\partial f_i/\partial x_j)$ at $p$ is nonzero, where $c = \dim P(V) - \dim_p X(V)$ and $x_j \in P(V)^*$.

\subsection{Semi-functorial resolution of singularities}
Let $X$ be a $\bVec$-variety. By functoriality, for every $V \in
\bVec$, we have an algebraic group action of $\oGL(V)$ on $X(V)$.
Hironaka's theorem guarantees that for every $V\in \bVec$,
there exists a resolution of singularities of $X(V)$. In fact, for each $V\in
\bVec$ such a resolution can be chosen to be $\oGL(V)$-equivariant (see e.g.~\cite{villamayor}), in the sense that $\oGL(V)$
acts algebraically on the resolution space and that the resolution map is
$\oGL(V)$-equivariant.
An interesting question is whether it is possible
to find a functorial way to assign to each vector space $V\in \bVec$
a resolution of singularities $\Omega(V)$ of $X(V)$ that is functorial
in $V$.  Some lead comes from the following example.

\begin{xmp}\label{xmp:resolutionsymmetric}
    For every $V \in \bVec$, denote with $S^2(V)$ the vector subspace
    of $V \otimes V$ of symmetric tensors (matrices), and observe that for every linear map $\phi: V \to V'$ we get a map $S^2(\phi): S^2(V) \to S^2(V')$ given as the restriction to $S^2(V)$ of the map $\phi \otimes \phi$.
    Let $X$ be the $\bVec$-variety inside $S^2$ of symmetric matrices of rank at most $r$.
    Then
    \[\Omega(V) \coloneqq \{(p, U)\mid p \in U \otimes U\} \subset S^2(V)\times \oGr(r, V)\]
    together with the projection $\pi_V$ on the first coordinate is a $\oGL(V)$-equivariant resolution of singularities of $X(V)$ for any vector space $V$ of dimension bigger than $r$ \cite[page 189]{weyman}.
    Moreover, for every injective linear map $\phi$ from $V$ to $V'$, we get a map $\Omega(\phi)$ from $\Omega(V)$ to $\Omega(V')$ defined by sending $(p, U)$ to $(S^2_\phi(p), \phi(U))$, and the following diagram commutes:
    \begin{center}
        \begin{tikzcd}
            \Omega(V) \arrow[r, "\Omega(\phi)"] \arrow[d, "\pi_V"] & \Omega(V') \arrow[d, "\pi_{V'}"]\\
            X(V) \arrow[r, "S^2(\phi)"] & X(V').
        \end{tikzcd}
    \end{center}
\end{xmp}

The above example scales back expectations.  First, we need to work
with Grassmannians that behave functorially only with respect to 
injective linear maps $\phi: V \to V'$.  Second, the dimension of $X(V)$
of Example~\ref{xmp:resolutionsymmetric} grows linearly in $\dim(V)$, so this is a resolution for one of the ``smallest''
$\bVec$-varieties only. Third, the variety $\Omega(V)$ is not an affine
variety, so even if it were functorial in $V$, it would not be a
$\bVec$-variety in our sense.

We plan to return to these issues in future work on not necessarily
affine varieties equipped with an action of the infinite general linear
group $\oGL$. The hope is that that setting is broad enough to allow for
$\oGL$-equivariant resolutions of singularities.

In the current paper, we will restrict ourselves to the following
situation. We say that a $\bVec$-variety $X$ is of {\em linear type}
if $\dim(X(V))$ is eventually a linear function of $\dim(V)$; see
Section~\ref{sec:lineartype} for details. Roughly speaking, these are
the varieties that can be parameterised by finitely many vectors.
The varieties in Examples~\ref{xmp:singularlocusboundedmatrices}
and~\ref{xmp:resolutionsymmetric} are of linear type. In
Section~\ref{sec:weakresolution} we construct an assignment $\Omega$
from $\bVec$ to finite-type schemes over $K$ that behaves functorially
with respect to suitable linear maps.
We further construct a transformation $\pi$ with the property that for every $V \in
\bVec$ the map $\pi_V: \Omega(V) \to X(V)$ is a {\em weak} resolution
of singularities such that for most linear maps $\phi: V \to V'$
the following diagram commutes:

\begin{center}
    \begin{tikzcd}
        \Omega(V) \arrow[r, dashed, "\Omega(\phi)"] \arrow[d, "\pi_V"] & \Omega(V') \arrow[d, "\pi_{V'}"]\\
        X(V) \arrow[r, "X(\phi)"] & X(V').    
    \end{tikzcd}
\end{center}

The adjective ``weak'' refers to the fact that the map $\pi_V$ is not in general birational on the whole smooth locus of $X(V)$: smooth points can have strictly positive-dimensional fibres.

\subsection{Relation to the existing literature}

Infinite-dimensional and $\oGL$-equivariant algebraic geometry
have recently gained much attention on account of their applications
to problems in commutative algebra such as Stillman's conjecture
\cite[Problem~3.14]{peeva-stillman:problems}. This conjecture was
first proved by Ananyan-Hochster \cite{Ananyan16} and later proved and
generalised using infinite-dimensional methods by Erman-Sam-Snowden
\cite{Erman18,Erman17b}; see \cite{Erman19} for an introduction to these
ideas. Similar ideas have found application in tensor decomposition.
For instance, Karam's insight that many notions of bounded rank can be
detected in small subtensors \cite{Karam22} was made more effective
for partition rank using ideas from infinite-dimensional geometry
\cite{Draisma23a}; \cite{bik-draisma-eggermont-snowden:uniformitylimits}
shows that tensors in the boundary of a tensor (de-)composition map can
be approximated in a uniform manner; and $\oGL$-equivariant geometry
also plays a crucial role in many of the open problems on tensors in
\cite{gesmundo:geometry}.

These developments have led to a systematic study of
infinite-dimensional varieties with an action of the infinite
general linear group; see \cite{bik-draisma-eggermont-snowden} in
characteristic zero and \cite{Bik24} in positive characteristic. The
general question is whether behaviour of finite-dimensional
algebraic varieties persists for infinite-dimensional ones if one
imposes $\oGL$-symmetry.  Successes were obtained for topological
Noetherianity \cite{draisma, bik-danelon-draisma:topologicalN},
Chevalley's theorem \cite{bik, bik-draisma-eggermont-snowden},
and ring-theoretic Noetherianity certain polynomial representations
of degree $\leq 2$ \cite{sam-snowden:GLmodulesoverpolyringsininfvars,
sam-snowden:GLmodulesoverpolyringsininfvarsII,nagpal-sam-snowden,sam-snowden:spequivariant}.
Excitingly, ring-theoretic Noetherianity was recently disproved in
positive characteristic \cite{Ganapathy24}; in characteristic zero it
remains open.

Theorem~\ref{thm:mainintro} fits in this line of research. It concerns
the singular loci of varieties depending functorially on the choice of a
vector space $V$.  We are certainly not the first to study the singular
loci of such varieties; here is a well-known example from the literature.

\begin{xmp}
Fix a positive integer $d \geq 3$, let $P$ be the polynomial functor $V
\mapsto V^{\otimes d}$, and let $Y$ be the closed $\bVec$-variety in $P$
defined by
\[ Y(V)=\{v_1 \otimes \cdots \otimes v_d \mid v_i \in V\}. \]
Define $X$ as the first secant variety of $Y$, i.e.,
\[ X(V)=\overline{\{p+q \mid p,q \in Y(V)\}}. \]
So $Y$ is the $\bVec$-variety of tensors of rank $\leq 1$, and $X$
is the $\bVec$-variety of tensors of border rank $\leq 2$. Then,
by \cite[Corollary 7.17]{Michalek15}, $\Xsing$ has $\binom{d}{2}$
irreducible components, namely,
\[ X_{12}(V)=\{A_{12} \otimes v_3 \otimes \cdots \otimes v_d \mid
A_{12} \in V \otimes V \text{ of rank } \leq 2 \text{ and } v_3,\ldots,v_d \in 
V\} \]
and its orbit under the symmetric group on $d$ letters. Note that here,
too, $\Xsing(V)$ is strictly larger than $\oSing(X(V))$ if $V$ is too
small. For instance, if $d=3$, then $X(K^2)$ is all of $K^2 \otimes
K^2 \otimes K^2$ and hence smooth, while $\Xsing(K^2)$ is nonempty. The results from \cite{Michalek15} were generalised from Segre
varieties to Segre-Veronese varieties in 
\cite{khadam-michalex-zwiernik:secant}. 
\end{xmp}

Our Main Theorem establishes the existence of such a uniform description
of the singular locus for any $\bVec$-variety. 
Furthermore, follow-up work
by Chiu-Danelon-Snowden relates our result to the more subtle notions
of singularity and smoothness in the realm of infinite-dimensional
varieties \cite{CDS}.

\subsection{Notation}
\begin{itemize}
    \item $\bbN$ denotes the set of natural numbers $\{0, 1, 2, \dots\}$,
    \item for $n \in \bbN$ we write $[n]\coloneqq \{1,\ldots,n\}$; in particular,
    $[0]=\emptyset$,
    \item $K$ denotes a field of characteristic zero,
    \item $\oGL_n$ is the general linear group on $K^n$, regarded as an algebraic group,
    \item $\bbA^n$ denotes the affine space over $K$ of dimension $n$,
    \item $\bbP^n$ denotes the projective space over $K$ of dimension $n$,
    \item a {\em variety} is a reduced affine scheme of finite type over $K$,
    \item $\oGr(d, V)$ denotes the Grassmannian of $d$-dimensional linear subspaces of the vector space $V$,
    \item $K[X]$ denotes the coordinate ring of the variety $X$,
    \item $\bSch$ denotes the category of schemes over $K$,
    \item $\bVec$ denotes the category of finite-dimensional vector spaces over $K$, and for another field $L$ we denote with $\bVec_L$ the category of finite-dimensional vector spaces over $K$.
\end{itemize}

\section{Background}\label{sec:Background}

\subsection{Polynomial functors and $\bVec$-varieties}

Let $K$ be a field of characteristic zero and recall that $\bVec$ is the category of finite-dimensional $K$-vector spaces with $K$-linear maps.
Let $V,W \in \bVec$. A {\em polynomial map}
from $V$ to $W$ is an element of $\oSym_K^\bullet(V^*)\otimes W$,
where $\oSym_K^\bullet(V^*)$ denotes the symmetric $K$-algebra on the dual space of $V$.

More concretely, let $\{x_1, \dots, x_n\}$ be the dual of a basis $\{v_1, \dots, v_n\}$ of $V$, and let $\{w_1, \dots, w_m\}$ be a basis for $W$.
An element $f$ of $\oSym^\bullet_K(V^*)\otimes W$ is of the form:
\begin{displaymath}
    f \coloneqq \sum_{i = 1}^m p_{i}(x_1, \dots, x_n) \otimes w_i,
\end{displaymath}
where the $p_i$ are polynomials in $K[x_1, \dots, x_n]$.
Then $f$ defines the (polynomial) map $V \to W$ mapping the vector $v = \sum_{j=1}^n \lambda_jv_j$ with $\lambda_i\in K$ to the vector $\sum_{i = 1}^m p_i(\lambda_1, \dots, \lambda_n)w_i$.
The {\em degree} of $f$ is the maximum among the degrees of the $p_i$ (and $-\infty$ if $f=0$) and $f$ is said {\em homogeneous of degree $d$} if all the $p_i$ are homogeneous of degree $d$.
As we are working over an infinite field, the assignment $\oSym^\bullet_K(V^*) \otimes W \to \operatorname{Map}(V,W)$---the maps from $V$ to $W$---is injective.

\begin{dfn}
A {\em polynomial functor} is a functor $P$ from the category $\bVec$
to itself such that for all vector spaces $U,V \in \bVec$ the assignment
\begin{displaymath}
 P_{U,V}: \oHom(U,V) \to \oHom(P(U), P(V))
\end{displaymath}
is a polynomial map whose degree is bounded from above by some element in
$\{-\infty,0,1,\ldots\}$ independent of $U,V$; the smallest such element
is called the {\em degree} of $P$ and denoted $\deg(P)$.  If for every
$U,V$ the map $P_{U,V}$ is homogeneous of a fixed degree $d$, we say
that the polynomial functor $P$ is {\em homogeneous of degree $d$}.
To simplify notation we write $P(\phi)$ instead of $P_{U,V}(\phi)$
when the domain $U$ and the codomain $V$ of a map $\phi$ are clear from
the context.
\end{dfn}

Many texts on polynomial functors allow that the degree of $P_{U,V}$
goes to infinity for $\dim(U),\dim(V)$ tending to infinity, so that
for instance
\[ V \mapsto \bigwedge^1 V  \oplus \bigwedge^2 V \oplus \cdots \]
is a polynomial functor. But we are only interested in polynomial functors
for which this degree can be bounded uniformly in $U,V$,
as required in the definition above.

\begin{dfn}
Let $P$ and $Q$ be polynomial functors.
A {\em polynomial transformation} $\alpha: P \to Q$ of polynomial functors is given by a polynomial map $\alpha_V : P(V) \to Q(V)$ for every $V \in \bVec$ such that for every $\phi \in \oHom(U,V)$ the diagram
\begin{center}
    \begin{tikzcd}
    P(U) \arrow[d, "P(\phi)" swap] \arrow[r, "\alpha_U"] & Q(U) \arrow[d, "Q(\phi)"] \\
    P(V) \arrow[r, "\alpha_V"] & Q(V)
    \end{tikzcd}
\end{center}
commutes.
\end{dfn}

\begin{xmp} \label{xmp:slicerank}
Many varieties of tensors are (unions of) images (or image closures)
of polynomial transformations. For instance, consider the functors $Q:V \mapsto V^{\otimes d}$ and $P:V \mapsto (V \oplus V^{\otimes d-1})^r$. Fix an $r$-tuple $(i_1,\ldots,i_r) \in [d]^r$, and consider the polynomial transformation $\alpha:P \to Q$ given by 
\[ \alpha_V((v_1,T_1),\ldots,(v_r,T_r))\coloneqq \sigma_1(v_1 \otimes T_1) + \cdots +
\sigma_r(v_r \otimes T_r) \]
where $\sigma_k :V^{\otimes d} \to V^{\otimes d}$ is the linear map
induced by
\[ u_1 \otimes \cdots \otimes u_d \mapsto 
u_2 \otimes \cdots \otimes u_{i_k} \otimes u_1 \otimes u_{i_k+1}
\otimes \cdots \otimes u_d.\]
A tensor in the image of $\alpha$ is a sum of $r$ tensors that factor
as a vector in one copy of $V$ times a tensor in the remaining
$(d-1)$-fold tensor power of $V$. Taking the union over all $r$-tuples
$(i_1,\ldots,i_d)$, we obtain the variety of tensors with {\em slice
rank} at most $r$ \cite{Tao16}.
\end{xmp}

\begin{rmk}
Polynomial functors with homogeneous polynomial transformations of
degree one form an Abelian category $\bPF$. On the other hand, the
category $\bPF^{\text{pol}}$ of polynomial functors equipped with
all polynomial transformations is not Abelian; e.g., the image of such a
transformation may no longer be a polynomial functor.
For us, a {\em subfunctor} $Q$ of a polynomial functor $P$ is a subobject of $P$ in the category $\bPF$.
\end{rmk}

\begin{rmk}\label{rmk:sumofhomogeneous}
In $\bPF$, a polynomial functor $P$ is the direct sum of its {\em
homogeneous components} defined by 
\begin{displaymath}
 P_i(V) \coloneqq \{\,p \in P(V) : P(\lambda \oid_V)(p) = \lambda^i p \text{
 for every } \lambda \in K\,\}.
\end{displaymath}
Indeed, each $P_i$ is a homogeneous polynomial functor of degree $i$, 
and we have $P = \bigoplus_{i \geq 0} P_{i}$, where only finitely many of the $P_i$
are nonzero (\cite[Section~2]{friedlander-suslin}). We note that $P_0$ is a constant polynomial functor, which assigns a fixed vector space $P(0) \in \bVec$ to all $V \in \bVec$ and the identity map to each linear map.
We call $P$ {\em pure} if $P_0=\{0\}$.
\end{rmk}

\begin{dfn}\label{dfn:irreducibilityGL}
    We say that a polynomial functor $P$ is {\em irreducible} if $P$
    has precisely one nonzero subobject in $\bPF$ (namely, $P$
    itself).
\end{dfn}

\begin{rmk}\label{dfn:irreducibilityPF}
    A polynomial functor $P$ is irreducible if and only if for every
    $V \in \bVec$ we have that $P(V)$ is either zero or an irreducible
    $\oGL(V)$-representation; this follows from 
    \cite[Lemma~3.4]{friedlander-suslin}.
\end{rmk}

\subsubsection{A well-founded order}\label{sssec:orderpolyfun}
(Isomorphism classes of) polynomial functors are partially
ordered by the relation $\prec$ defined by $Q \prec P$ if $Q \not
\cong P$ and for the largest $e$ with $Q_e \not \cong P_e$ the former
is a quotient of the latter.  This partial order is well-founded; see
\cite[Lemma 12]{draisma} for a proof.

\subsubsection{Connection to Schur functors}
We refer to \cite{fulton-harris:reprfirstcourse} for background on
Schur functors. 
A {\em partition} $\lambda$ of {\em length $n$} is a sequence $(\lambda_1, \lambda_2, \dots, \lambda_n)$ of integers such that $\lambda_i \geq \lambda_{i+1}>0$.
It corresponds to the Young diagram where the $i$-th row has $\lambda_i$ boxes.
To each $\lambda$ one associates an irreducible representation $U_\lambda$ of the symmetric group $\oSym([d])$ with $d=|\lambda|\coloneqq \sum_i \lambda_i$. The {\em Schur functor} $\bbS_\lambda$ is then defined as $\bbS_\lambda(V) \coloneqq \oHom_{\oSym([d])}(U_\lambda,V^{\otimes d})$. 

\begin{xmp}\label{xmp:skewsymmetricpower}
    If $\lambda$ is the partition $(1, \dots, 1)$ of length $d$, then
    $\bbS_\lambda = \Lambda^d$, the $d$-th skew-symmetric power. At the
    other extreme, if $\lambda$ is the partition $(d)$ of length $1$,
    then $\bbS_\lambda = S^d$, the $d$-th symmetric power. General
    Schur functors are constructed via a combination of symmetrisation
    and skew-symmetrisation. 
\end{xmp}

\begin{rmk}
    Since $K$ has characteristic zero, the Schur functors are the simple
    objects in $\bPF$, and moreover every polynomial functor is a finite direct
    sum of Schur functors. This follows from the corresponding statement for polynomial representations of $\oGL_n$ and \cite[Lemma 3.4]{friedlander-suslin}.
\end{rmk}

\subsubsection{Shifting}\label{sssec:shiftoperation}
\begin{dfn}
Given a finite-dimensional vector space $U$, we define the \textit{shift functor} $\oSh_U : \bVec \to \bVec$ by assigning to each $V \in \bVec$ the vector space $U \oplus V$, and to each map $\phi \in \oHom(V, W)$ the map $\oid_U \oplus \phi$. For a polynomial functor $P$, we denote with $\oSh_UP$ the composition $P\circ \oSh_U$ and call it 
the {\em shift of $P$ by $U$.} This is a polynomial functor  assigning to each $V \in \bVec$ the vector space $P(U \oplus V)$, and to a morphism $\phi \in \oHom(V, W)$ the morphism $P(\oid_U \oplus \phi)$.
\end{dfn}

\begin{rmk}\label{rmk:shiftissmaller}
    Let $P$ be a polynomial functor. Then
\begin{displaymath}
    \oSh_U P \cong P \oplus P'
\end{displaymath}
with $P'$ a polynomial functor of strictly lower degree than $P$ (note that
$P'\cong (\oSh_U P)/P$). Indeed, the natural transformation 
$P \to \oSh_U P$ given by 
\[ P(0_{V \to U} \oplus \id_V): P(V) \to P(U \oplus V) \] 
has
as a left inverse the natural transformation $\oSh_U P \to P$ given by
\[ P(0_{U \to V} \oplus \id_V):P(U \oplus V) \to P(V); \] 
and for the
statement about the degree see \cite[Lemma 14]{draisma}.
\end{rmk}

\begin{rmk}\label{rmk:boundedlength}
    Let $\lambda$ be a partition, $U,V$ finite-dimensional vector spaces, and consider the Schur functor $\bbS_\lambda$.
    Then we have:
    \[
    (\oSh_U \bbS_{\lambda})(V)=\bbS_{\lambda}(U \oplus V) = \bigoplus_{\mu,\nu}\left (\bbS_\mu
    (U) \otimes \bbS_\nu(V) \right )^{ N_{\mu \nu\lambda} },
    \]
    where the direct sum ranges over all partitions $\mu$ and $\nu$,
    and the symbol $N_{\mu \nu \lambda}$ is the Littlewood-Richardson
    coefficient: the number of copies of $\bbS_\lambda$ in $\bbS_\mu
    \otimes \bbS_\nu$ \cite[Exercise~6.11]{fulton-harris:reprfirstcourse}.
    We will use later that $N_{\mu\nu\lambda}$ can be nonzero only when $|\lambda|=|\mu|+|\nu|$ and moreover the lengths of $\mu$ and $\nu$ are at most the length of
    $\lambda$.
    Furthermore, we have $N_{() \lambda \lambda}=1$, in accordance
    with Remark~\ref{rmk:shiftissmaller}.
\end{rmk}

\subsubsection{$\bVec$-varieties}

Let $\bSch$ denote the category of schemes over $K$.  A polynomial functor $P$ gives rise to a functor
$\bVec \to \bSch$ that sends $V$ to the spectrum of the symmetric algebra
on $P(V)^*$. We write $P$ for this functor $\bVec \to \bSch$, as well,
and we write $K[P(V)]$ for said symmetric algebra.

\begin{rmk}
    \label{rmk:base-change-pol-functor}
    For every $U,V \in \bVec$ the polynomial map
    \[
        P_{U,V} \colon \oHom(U,V) \to \oHom(P(U),P(V))
    \]
    extends naturally to a morphism of affine schemes, which we will still denote by $P_{U,V}$. If $R$ is any $K$-algebra, then on $R$-points $P_{U,V}$ is given by a map
    \[
        P_{U,V}^R \colon \oHom_{\bMod_R}(U \otimes_K R, V \otimes_K R) \to \oHom_{\bMod_R}(P(U)\otimes_K R, P(V) \otimes_K R).
    \]
    Via the embedding $M \in \bMod_R \mapsto
    \oSpec(\oSym_R^\bullet(M^*)) \in \bSch$ maps on the right hand side
    are identified with $R$-morphisms $P(U) \times_K \oSpec R \to P(V)
    \times_K \oSpec R$.
\end{rmk}

\begin{dfn}
Let $X, Y: \bVec \to \bSch$ be functors.
Let $\alpha: X \to Y$ be a natural transformation.
We say that $\alpha$ is a {\em closed embedding} if $\alpha_V: X(V) \to Y(V)$ is a closed embedding for every $V \in \bVec$.
We say that a functor $X: \bVec \to \bSch$ admitting a closed
embedding $\alpha: X \to P$ in some polynomial functor $P$ is an {\em affine $\bVec$-scheme}.
The category of affine $\bVec$-schemes is the full subcategory  (whose objects are affine $\bVec$-schemes) in the functor category $\bSch^\bVec$.
\end{dfn}

The following is our main object of study: it is our notion of variety in polynomial functors.\label{dfn:Vecvariety}
\begin{dfn}
Let $X$ be an affine $\bVec$-scheme.
If $X(V)$ is a reduced affine scheme for every $V\in \bVec$, then we say that $X$ is a {\em $\bVec$-variety}.
\end{dfn}
We can think of a $\bVec$-variety $X$ in $P$ as a functor $X: \bVec
\to \bSch$ such that $X(V)\subset P(V)$ is a closed subvariety and $X(\phi) = P(\phi)_{|_{X(U)}}$ for every $\phi \in \oHom(U,V)$.

Let $X$ and $Y$ be $\bVec$-varieties.
A {\em morphism} of $\bVec$-varieties is a natural transformation
$\alpha: X \to Y$, i.e.~given by the data of a morphism $\alpha_V:
X(V) \to Y(V)$ of affine varieties for every $V \in \bVec$ such that for every $\phi \in \oHom(U,V)$ we have $Y(\phi)\circ \alpha_U = \alpha_V \circ X(\phi)$.
The category of $\bVec$-varieties is the full subcategory, in the category of affine $\bVec$-schemes, whose objects are $\bVec$-varieties.
The following proposition is immediate.

\begin{prp} \label{prp:Easy}
Let $X$ be a closed $\bVec$-variety in a polynomial
functor $P$ and let $\phi \in \oHom(V,W)$. Then:
\begin{enumerate}
\item if $\phi$ is injective, then $P(\phi)$ restricts to a
closed embedding $X(V) \to X(W)$;
\item if $\phi$ is surjective, then $P(\phi)$ restricts to a
surjective morphism $X(V) \to X(W)$; and 
\item if $V=W$ and $\phi$ is a linear isomorphism, then $P(\phi)$
restricts to an automorphism $X(V) \to X(V)$---and indeed, the map $\oGL(V)
\times X(V) \to X(V), (\phi,p) \mapsto P(\phi)p$ is an algebraic group
action.
\end{enumerate}
\end{prp}
\begin{proof}
We prove the first item; the rest is proved in a similar fashion. If
$\phi$ is injective, then let $\psi \in \oHom(W,V)$ be such that
$\psi \circ \phi=\oid_V$. Then $P(\psi) \circ P(\phi)=P(\psi \circ
\phi)=P(\oid_V)=\oid_{P(V)}$ by functoriality. It follows that $P(\phi)$
is an injective linear map, hence defines a closed embedding $P(V)
\to P(W)$, and this restricts to a closed embedding $X(V)
\to X(W)$.
\end{proof}

\begin{rmk}
    \label{rmk:extendingbasefield}
    Let $P$ be a polynomial functor over $K$ and $L$ be a field extension of $K$. Following \cite[Section 2.8]{draisma}, we construct a polynomial functor $P_L \colon \bVec_L \to \bVec_L$ as follows. For each $U \in \bVec_L$ we make a choice of $V_U \in \bVec_K$ and isomorphism $\psi_U \colon U \to V_U \otimes_K L$ in $\bVec_L$. In particular, if $U = V \otimes_K L$ for $V\in \bVec_K$ we may choose $U_V = V$ and $\psi_U = \id$. Then we set $P_L(U) \coloneqq P(V_U) \otimes_K L$. 
    Moreover, for $\widetilde{\varphi} \in \oHom_{\bVec_L}(U',U)$ we define
    \[
        P_L(\widetilde{\varphi}) \coloneqq P_{V_{U'},V_U}^L(\psi_U \circ \widetilde{\varphi} \circ \psi_{U'}^{-1}),
    \]
    see Remark~\ref{rmk:base-change-pol-functor}. Note that for $\varphi \in \oHom_{\bVec_K}(V',V)$, by our choice above, we have that $P_L(\varphi \otimes_K L) = P(\varphi) \otimes_K L$.

    Thinking of $P$ as a functor $\bVec_K \to \bSch_K$, for every $\varphi \in \oHom_{\bVec_K}(V',V)$ we obtain the following diagram
    \[
        \begin{tikzcd}[column sep=huge]
            P_L(V' \otimes_K L) = P(V') \times_K \oSpec L \arrow[d] \arrow[r, "P(\varphi) \times_K \oSpec L"] & P(V) \times_K \oSpec L = P_L(V \otimes_K L) \arrow[d] \\
        P(V') \arrow[r, "P(\varphi)"] & P(V).
        \end{tikzcd}
    \]
    If $X$ is a $\bVec$-subvariety of $P$, then we define the base change $X_L$ as a $\bVec_L$-subvariety of $P_L$ via $X_L(U) \coloneqq X(V_U) \times_K \oSpec L$, which comes with a closed immersion into $P_L(U)$.
\end{rmk}

\begin{rmk}
    \label{rmk:image-of-hom}
    If $X$ is a $\bVec$-subvariety of $P$ and $V,W \in \bVec$, then we have a morphism of schemes
    \[
        \oHom(V,W) \times X(V) \to X(W).
    \]
    Indeed, by Remark~\ref{rmk:base-change-pol-functor}, for any $K$-algebra $R$ this morphism is given on $R$-points as
    \[
        \oHom_{\bMod_R}(V\otimes_K R,W \otimes_K R) \times X(V)(R) \to X(W)(R),\: (\varphi,x) \mapsto P_{V,W}^R(\varphi)(x).
    \]
    In particular, for $R = L$ a field extension of $K$, the morphism maps $(\varphi,x)$ to $P_L(\varphi)(x)$.
\end{rmk}

\begin{lm}
    \label{lm:intersection-subspace}
    Let $X$ be a $\bVec$-subvariety of a polynomial functor $P$ and $V\in \bVec$. 
    \begin{enumerate}
        \item If $U$ is a subspace of $V$, then $X(U) = X(V) \times_{P(V)} P(U)$.
        \item If $U_1,U_2$ are subspaces of $V$, then
            \[
                X(U_1\cap U_2) = X(U_1) \times_{X(V)} X(U_2) = X(U_1) \times_{P(V)} P(U_2) = P(U_1) \times_{P(V)} X(U_2).
            \]
    \end{enumerate}
\end{lm}

\begin{proof}
    For the first assertion, let $\iota_{U,V} \colon U \to V$ be the
    inclusion and consider a diagram in which the two squares with
    solid arrows commute:
    \[
        \begin{tikzcd}[column sep = large]
            Y \arrow[rrd, bend left,"g_{P(U)}"] \arrow[rdd, bend right,"g_{X(V)}"']  \arrow[rd, dashed, "g"]& & \\
             & X(U) \arrow[d, "X(\iota_{U,V})"'] \arrow[r, "\iota_{X(U)}"] & P(U) \arrow[d, "P(\iota_{U,V})"]\\
             & X(V) \arrow[r, "\iota_{X(V)}"'] & P(V), 
        \end{tikzcd}
    \]
    where we have to show the existence of a unique $g$ making the
    two triangles commute. Choose a retraction $\pi_{U,V} \colon V \to U$
    and set $g \coloneqq X(\pi_{U,V}) \circ g_{X(V)}$. Then
    \begin{multline*}
        \iota_{X(V)} \circ X(\iota_{U,V}) \circ g  = P(\iota_{U,V}) \circ \iota_{X(U)} \circ X(\pi_{U,V}) \circ g_{X(V)} = P(\iota_{U,V}) \circ P(\pi_{U,V}) \circ \iota_{X(V)} \circ g_{X(V)} \\
         = P(\iota_{U,V}) \circ P(\pi_{U,V}) \circ P(\iota_{U,V}) \circ g_{P(U)} = P(\iota_{U,V}) \circ g_{P(U)} = \iota_{X(V)} \circ g_{X(V)}.
    \end{multline*}    
    and since $\iota_{X(V)}$ is a monomorphism we get that $g_{X(V)} = X(\iota_{U,V}) \circ g$. Furthermore
    \[ \iota_{X(U)} \circ g = \iota_{X(U)} \circ X(\pi_{U,V}) \circ
    g_{X(V)}=P(\pi_{U,V}) \circ \iota_{X(V)} \circ g_{X(V)}
    =P(\pi_{U,V}) \circ P(\iota_{U,V}) \circ g_{P(U)}
    =g_{P(U)}, \]
    as desired.  
    For the second assertion, note that by (1) we only have to prove the first equality. For that, we proceed similarly as before. We write $W \coloneqq U_1 \cap U_2$ and $\iota_{U_i,V}$, $\iota_{W,U_i}$ and $\iota_{W,V}$ for the inclusions in $\bVec$. Now choose retractions $\pi_{U_i, V}$, $\pi_{W,U_i}$ and $\pi_{W,V}$ such that $\pi_{W,V} = \pi_{W,U_1} \circ \pi_{U_1,V} = \pi_{W,U_2} \circ \pi_{U_2,V}$ and
    \[
        (\iota_{U_i,V} \circ \pi_{U_i,V}) \circ (\iota_{U_j,V} \circ \pi_{U_j,V}) = \iota_{W,V} \circ \pi_{W,V}, \: i\neq j.
    \]
    Let $Y \in \bSch$ and $g_{U_i} \colon Y \to X(U_i)$ a pair of
    morphisms satisfying $X(\iota_{U_1,V}) \circ g_{U_1} =
    X(\iota_{U_2,V}) \circ g_{U_2}$. Then $g \coloneqq X(\pi_{W,U_1})
    \circ g_{U_1} = X(\pi_{W,U_2}) \circ g_{U_2}$ defines the unique morphism $Y \to X(W)$ making the corresponding diagram commutative. 
\end{proof}

As an immediate consequence we get the following. Note that, if $p$ is an $L$-point of $X$ and $L'$ a field extension of $L$, then we get a natural lift of $p$ to an $L'$-point of $X \times_L \oSpec L'$, which we still denote by $p$.

\begin{lm}
    \label{lm:minimaldimension}
    Let $X$ be a $\bVec$-variety, let $V \in \bVec$, and let $L$ be an
    extension field of $K$. Then for any $L$-valued point $p$ of $X(V)$
    there exists a unique minimal $L$-subspace $U$ of $V \otimes_K L$
    with $p \in X_L(U)$.
\end{lm}

We may, in particular, take an arbitrary scheme-theoretic point $p$ of
$X(V)$ and take $L$ to be its residue field. For this particular choice
of $L$ we write $U_p \subset V \otimes_K L$ for the unique minimal subspace $U$
in the lemma.

\begin{proof}
Clearly there exist $L$-subspaces $U$ with $p \in X_L(U)$, such as
$U=V \otimes_K L$. If $U_1,U_2$ are two such subspaces, then $p$ gives rise
to morphisms $\oSpec(L) \to X_L(U_i)$ such that a square diagram
with maps $X_L(U_i) \to X_L(V \otimes L)$ commutes. By 
Lemma~\ref{lm:intersection-subspace}, part (2) applied to $X_L$, we then
find that $p \in X_L(U_1 \cap U_2)$. This implies that there is a unique
minimal $U$ with $p \in X_L(U)$.
\end{proof}

\begin{xmp}
    Let $X$ be the $\bVec$-subvariety of $S^2$ of quadratic forms of
    rank at most $1$. For $V = K^2$ we have $P(V) = \langle e_1^2,
    e_1 e_2, e_2^2 \rangle$ with corresponding coordinates $a,b,c$,
    and $X(V)$ is the irreducible hypersurface defined by the equation
    $b^2-4ac=0$. The generic point $\eta$ of $X(V)$ has as residue field
    $L$ the fraction field of $K[a,b,c]/(b^2-4ac)$, and 
    \[ \eta=a e_1^2 + b e_1 e_2 + c e_2^2 = \frac{1}{4a} (2ae_1 + b e_2)^2. \]
    So $U_\eta$ is the one-dimensional subspace of $V \otimes L$
    spanned by $2ae_1 + b e_2$. 
\end{xmp}

An important consequence of working in characteristic zero is the
following.

\begin{lm}
In the setting of Lemma~\ref{lm:minimaldimension}, if $U_{p,L}$ is
the minimal subspace promised by that Lemma, $L'$ is an extension
field of $L$, and $U_{p,L'} \subseteq V \otimes L'$ is the minimal
subspace promised by that lemma when regarding $p$ as an $L'$-point,
then $U_{p,L'}=U_{p,L} \otimes_L L'$.
\end{lm}

\begin{proof}
By Lemma~\ref{lm:intersection-subspace}, it suffices to prove this for
$X=P$. First consider the case where $P(V)=V^{\otimes d}$. Then for
every $i=1,\ldots,d$, $p$ can be regarded as an $L$-linear map $p_i:((V
\otimes_K L)^*)^{\otimes [d] \setminus \{i\}} \to V \otimes_K L$,
and $U_{p,L}$ is the sum of the images of all these linear maps. This
clearly has the property that $U_{p,L'}=U_{p,L} \otimes_L L'$.

Next we claim that a polynomial functor $P=P_1 \oplus P_2$
satisfies the statement of the lemma if and only if $P_1$ and $P_2$
do. Indeed, for ``if'' one verifies that for $p=(p_1,p_2)$ we have
$U_{p,L}=U_{p_1,L}+U_{p_2,L}$, and for ``only if'' one verifies that
$U_{p_1,L}=U_{(p_1,0),L}$ and $U_{p_2,L}=U_{(0,p_2),L}$. 

Since any polynomial functor is a direct sum of irreducible summands
(Schur functors) of functors of the form $V \mapsto V^{\otimes d}$,
the lemma follows. 
\end{proof}

\begin{rmk}
We warn the reader that the lemma is not true
in characteristic $q$. For instance, take $K=L=\FF_q(t)$ and let $L'$
be an algebraic closure of $K$. Then the the element $p=e_1^q +
t e_2^q \in S^q(K^2)$ would have $U_{p,L}=\langle e_1,e_2 \rangle$,
while $U_{p,L'}=\langle e_1 + t^{1/q} e_2 \rangle$.
\end{rmk}

Let $P$ be a polynomial functor and $V \in \bVec$ of dimension $n$.
For $d\leq n$ let $e \coloneqq \dim P(K^d)$ (see also
Section~\ref{sssec:dimension}). This data gives rise to a morphism $\oGr(d,V) \to \oGr(e,P(V))$. Indeed, on any standard affine open of $\oGr(d,V)$, elements are represented by $n \times d$-matrices and get mapped to a standard affine open of $\oGr(e,P(V))$ via the (polynomial) map
\[
    \oHom(K^d,K^n) \to \oHom(P(K^d),P(K^n)).
\]
Write $\Gamma_e(W) \subset W \times \oGr(e,W)$ for the classical incidence variety given on $K$-points by
\[
    \{(p,U) \in W \times \oGr(e,W) \mid p \in U\}.
\]

\begin{dfn}
    \label{dfn:incidence-variety}
    Let $P$ be a polynomial functor and $V \in \bVec$. The $d$-th \emph{incidence variety} on $P$ is defined as 
    \[
        \Gamma_{d,P}(V) \coloneqq \Gamma_e(P(V)) \times_{\oGr(e,P(V))} \oGr(d,V).
    \]
    Its image $\Lambda_{d,P}(V)$ in $P(V)$ is closed and defines the $d$-th \emph{subspace variety}. Similarly, if $X$ is a closed $\bVec$-subvariety of $P$, then we define
    \[
        \Gamma_{d,X}(V) \coloneqq \Gamma_{d,P}(V) \times_{P(V)} X(V), \: \Lambda_{d,X}(V) \coloneqq \Lambda_{d,P}(V) \times_{P(V)} X(V)
    \]
    and call it the $d$-th incidence (resp.\ subspace) variety of $X$.
\end{dfn}

By the above, for every field extension $L \subset K$, we have
\[
    \Lambda_{d,X}(V)(L) = \{p \in X(V)(L) \mid \exists U \in
    \oGr(d,V\otimes_K L): p \in X_{L}(U)\}. 
\]
In particular, the assignment $V \mapsto \Lambda_{d,X}(V)$ defines a $\bVec$-subvariety of $X$, denoted by $\Lambda_{d,X}$. In contrast, the assignment $V \mapsto \Gamma_{d,X}(V)$ is only compatible with injective maps $U \to V$ and does not define a $\bVec$-variety. See also Example~\ref{xmp:resolutionsymmetric} and Example~\ref{xmp:incidence-variety}.

\begin{lm}\label{lm:UpSemicont}
The function $p \mapsto \dim(U_p)$ from $X(V)$ to $\{0,\ldots,\dim(V)\} \subset \bbN$ is lower semicontinuous.
\end{lm}

\begin{proof}
     The preimage of $\{0,\ldots,d\}$ is the (closed) subspace variety $\Lambda_{d,X}(V)$.
\end{proof}

\comment{
\begin{dfn}\label{dfn:minimaldimension}
    We say that a point $p\in X(V)$ is {\em defined in dimension $k$} is the vector space $U$ of Lemma~\ref{lm:minimaldimension} has dimension $k$.
\end{dfn}

\Jcom{Not very appealing terminology. Do we really need this?}
}

\begin{dfn}
Let $\alpha: P \to Q$ be a polynomial transformation.
We define the \textit{image} of $\alpha$ to be the functor
$\oIm(\alpha)$ from $\bVec$ to sets that assigns $\alpha_V(P(V))$ to every $V\in \bVec$ and the maps $Q(\phi)_{|_{\alpha_U(P(U))}}$ to every $\phi \in \oHom(U,V)$.
By taking the closure of $\alpha_V(P(V))$ inside $Q(V)$, we define the \textit{closure of the image} of $\alpha$ and denote it by $\overline{\oIm\alpha}$.
A straightforward verification shows that $\overline{\oIm\alpha}$ is a closed $\bVec$-subvariety of $Q$.
\end{dfn}

\begin{dfn}
    Let $X, Y$ be $\bVec$-varieties.
    A morphism $\alpha: X \to Y$ is {\em dominant} if the morphism $\alpha_V$ is dominant on $Y(V)$ for every $V \in \bVec$.
\end{dfn}

\begin{dfn} \label{dfn:localisation}
Let $X$ be a closed $\bVec$-subvariety in the polynomial functor $P$,
and let $h \in K[P(0)]$. Then $X[1/h]$ is the closed $\bVec$-subvariety in the
polynomial functor $Q:V \mapsto K \oplus P(V)$ defined as
\[ X[1/h](V)=\{(t,p) \mid t \cdot h(p) =1\}, \]
where we regard $h$ as a function on $P(V)$ via pullback along the
linear map $P(0_{V \to 0})$. We regard $X[1/h]$ as an open $\bVec$-subvariety of $X$.
\end{dfn}

\begin{dfn}
Let $X$ be a closed $\bVec$-subvariety in the polynomial functor $P$
and let $U \in \bVec$. Then $\oSh_U X$ is the closed $\bVec$-subvariety
of $\oSh_U P$ defined by $(\oSh_U X)(V)\coloneqq X(U \oplus V) \subset P(U
\oplus V)=(\oSh_U P)(V)$. 
\end{dfn}

\comment{
\subsection{Infinite-dimensional $\oGL$-varieties}\label{ssec:polfuninverselimit}
We refer to \cite{bik-draisma-eggermont-snowden} for the notion of $\oGL$-variety.
We just recall that a $\oGL$-variety is an affine scheme that is the spectrum of a finitely $\oGL$-generated $\oGL$-algebra, namely, an algebra admitting an action of $\oGL$ by algebra automorphisms (turning it into a polynomial $\oGL$-representation), and such that it is generated as a $K$-algebra by the orbits under $\oGL$ of a finite number of elements.

\begin{xmp}
    Consider the space of $\bbN\times \bbN$-matrices whose ring of global section is the polynomial ring $R = K[x_{i,j} \mid i,j \in \bbN]$.
    The usual action of $\oGL$ given by simultaneous row and column operations turns $R$ into a $\oGL$-algebra.
    Note that it is generated by $1$ and $x_{1,2}$.
    Let $I$ be the $\oGL$-stable ideal of $R$ given by the $r \times r$-minors of the matrix $(x_{i,j})_{i,j = 1}^\bbN$.
    Then $\oSpec(R/I)$ is a $\oGL$-variety.
\end{xmp}

Let $X$ be a $\bVec$-variety.
For $n \in \bbN_{\geq 1}$ define the map $\pi_n: K^{n} \to K^{n-1}$ to be the projection onto the first $n-1$ components.
The collection $X(K^n)$ and the morphisms $ X(\pi_{n})$ form an inverse system.
Denote with $X_\infty$ its inverse limit.
Explicitly, up to isomorphism one has:
\begin{displaymath}
 X_\infty = \left \{(\,p_n)_{n \in \bbN} \in \prod_{n \in \bbN} X(K^n) \mid X(\pi_j)(\,p_{j}) = p_{j-1} \text{ for all } j \in \bbN_{\geq 1} \right \},
\end{displaymath}
and the natural projections $\pi_{\infty, n}: X_\infty \to X(K^n)$ are given by mapping $(p_{i})_{i \in \bbN}$ to $p_n$.
The inverse limit $X_\infty$ is a topological space with respect to the inverse limit topology that coincides with the Zariski topology induced by the ring $\varinjlim_n K[X(K^n)]$ (the direct limit of $K[X(K^n)]$ with the dual maps of $X(\pi_n)$).
If $X$ is a closed $\bVec$-variety of $P$ and $I_n$ denotes the defining ideal of $X(K^n)$ inside $K[P(K^n)]$, then $X_\infty$ is the subvariety of $P_\infty$ given by $\cV\left(\bigcup_n I_n\right)$ and it is stable under $\oGL$.
It is in particular a $\oGL$-variety.
Viceversa, given a polynomial functor $P$, and a closed $\oGL$-subvariety $Y$ in $P_\infty$, we can construct a $\bVec$-variety $X$ of $P$ such that its inverse limit $X_\infty$ satisfies $X_\infty = Y$.
}

\subsection{The embedding Theorem and its consequences}\label{sec:embeddingtheorem}
We collect several results about $\bVec$-varieties.  The following theorem
is implicit in \cite{draisma}, and it was made explicit in the infinite-dimensional setting in \cite{bik-draisma-eggermont-snowden}.

\begin{thm}[Embedding Theorem]\label{thm:Embedding}
Let $P$ be a polynomial functor and let $X$ be a proper $\bVec$-variety of $P$.
Let $R$ be an irreducible subfunctor of $P$ of degree $>0$ and let $\pi : P \to P/R$ be the projection transformation.
Let $X'$ be the closure of the projection of $X$ along $\pi$.
Then one of the following holds:
\begin{enumerate}
    \item\label{cilindrical} $X = \pi^{-1}(X')  \cong X' \times R$; or 
    \item\label{closedembedding} there exists a finite-dimensional
    vector space $U$ and a function $h \in K[P(U)]$ that does not
    vanish identically on  $X(U)$ such that the projection $\oSh_U P \to
    (\oSh_U P)/R$ restricts to a closed embedding 
    \[ (\oSh_U X)[1/h] \to ((\oSh_U P)/R)[1/h]. \] 
    Furthermore, $h$ can be chosen as the directional derivative
    $\frac{\partial f}{\partial r}$ for some vector $r \in R(U)$ and
    some function $f \in K[P(U)]$ that vanishes identically on $X(U)$,
    and for any $h$ constructed in this manner the map above is a closed
    embedding.
\end{enumerate}
\end{thm}

In the last statement, it may be that the partial derivative does vanish
identically on $X$, in which case, of course, the map is trivially a closed
embedding. 

\begin{rmk}\label{rmk:quotientissmaller}
In case (1) the isomorphism comes from the fact that $P \cong (P/R)
\oplus R$ as polynomial functors (since $K$ has characteristic zero). 
In case (2) we identify $R \subset P$ with a subobject of $\oSh_U P$
via Remark~\ref{rmk:shiftissmaller}, and we observe that $h$ is a
function on $(\oSh_U P)(0)$, so that Definition~\ref{dfn:localisation}
applies. 
\end{rmk}

We recap below some of the applications of the Embedding Theorem.

\subsubsection{Topological Noetherianity of polynomial functors.}
\label{sssec:Noetherianity} 

\begin{thm}[{\cite[Theorem 1]{draisma}}]\label{thm:Noetherian}
Let $X$ be a $\bVec$-variety. Then every descending chain of $\bVec$-subvarieties
\[X = X_0 \supset X_1 \supset X_2 \supset \ldots \]
stabilises, that is, there exists $N\geq0$ such that for each $n\geq N$ we have $X_n = X_{n+1}$. Equivalently, for any closed $\bVec$-subvariety $Y$ of $X$ there exist finitely many $U_1,\ldots,U_k \in \bVec$ and $f_i \in K[X(U_i)]$ such that for all $V \in \bVec$ we have 
\[Y(V)=\{p \in X(V) \mid \forall i \; \forall \phi \in \oHom(V,U_i): f_i(X(\phi)p)=0\}.\]
\end{thm}

In the latter setting, we say that the $f_i$ define the $\bVec$-variety $Y$ in $X$.

\begin{dfn}
A $\bVec$-variety $X$ is called {\em irreducible} if it is non-empty and
if $X=Y \cup Z$ for closed $\bVec$-subvarieties $Y,Z$ of $X$ implies
that $X=Y$ or $X=Z$. 
\end{dfn}

Irreducibility of $X$ is equivalent to the condition that $X(V)$ is
irreducible for all $V \in \bVec$. By Noetherianity, every $\bVec$-variety
can be written in a unique manner as $X_1 \cup \cdots \cup X_k$ where
the $X_i$ are irreducible and none is contained in any other; the $X_i$
are the {\em irreducible components} of $X$.

\subsubsection{The shift theorem}

Next we recall the shift theorem in \cite{bik-draisma-eggermont-snowden}.
We rephrase it here in the language of polynomial functors and $\bVec$-varieties.
\begin{prp}[{\cite[Theorem~5.1]{bik-draisma-eggermont-snowden}}]\label{prp:openisom}
Let $X$ be a $\bVec$-variety. Then there exist a finite-dimensional
vector space $U$, a nonzero function $h \in K[X(U)]$, a polynomial
functor $Q$, and a finite-dimensional affine variety $B$ over $K$ 
such that \[(\oSh_U X) [1/h] \cong B \times Q\;.\]
\end{prp}

For a proof in the functorial setting, see \cite[Proposition~2.4.4]{danelon:thesis}.
\begin{rmk}\label{rmk:smaller}
If $X \subset P$ and there is an $R$ in the top-degree part of
$P$ for which case (2) of Theorem~\ref{thm:Embedding} applies, then the
polynomial functor $Q$ of Theorem~\ref{prp:openisom} can be chosen so
that $Q \prec P$ in the order of polynomial functors given in Section~\ref{sssec:orderpolyfun}.
Indeed, $Q$ can be chosen isomorphic to a subfunctor of 
$\oSh_{U \oplus U'} P/ \oSh_{U'}R$ with $U$ as in case (2) and $U'$ a
second vector space. This functor is smaller than $P$.
\end{rmk}

\comment{
\subsubsection{Morphisms and $\bVec$-varieties} \label{ssec:Morphisms}

Let $X$ be an irreducible $\bVec$-variety. By
Proposition~\ref{prp:openisom} there exists a dominant morphism $B \times
Q \to \oSh_U X$ for some finite-dimensional variety $B$ and some pure
polynomial functor $Q$. By irreducibility of $X$, $B$ can be chosen
irreducible. Furthermore, the maps $X(0_{U \to V} \oplus \id_V):
X(U \oplus V) \to X(V)$ form a surjective morphism $\oSh_U X
\to X$. Hence every irreducible $\bVec$-variety $X$ admits a dominant
morphism $B \times Q \to X$ for some irreducible finite-dimensional
variety $B$ and some pure polynomial functor $Q$---$X$ is ``unirational
in the $\bVec$-direction''. See the unirationality theorem in
\cite{bik-draisma-eggermont-snowden}. 

\begin{dfn}[The $\bVec$-version of {\cite[Definition~8.1]{bik-draisma-eggermont-snowden}}]\label{dfn:typicalmorphism2}
   Let $B$ be a finite-dimensional variety, $Q$ a pure polynomial
   functor, and $X$ be an irreducible $\bVec$-variety.
    We say that a dominant morphism $\alpha: B\times Q \to X$ is {\em typical (for $X$)} if there does not exist a proper $\bVec$-subvariety $Z$ of $B \times Q$ such that $\alpha_{|_Z}$ is dominant on $X$.
\end{dfn}

\begin{prp}\label{prp:typicalisminimal}
    Let $X$ be an irreducible $\bVec$-variety and let $\alpha: B \times Q \to X$, and $\beta : B' \times Q' \to X$ be typical morphisms.
    Then $Q' \cong Q$ and $\dim(B') = \dim(B)$.
\end{prp}

\begin{proof}
After passing to the inverse limit, use \cite[Corollary~8.5]{bik-draisma-eggermont-snowden}.
\end{proof}

}

\subsubsection{Dimension}
\label{sssec:dimension}

The following is well-known.

\begin{prp}\label{prp:dimensionpolyfun}
Let $P$ be a polynomial functor. Then there exists a univariate polynomial $f_P$ in $\dim(V)$ of degree $\deg P$ with rational coefficients such that for every $V \in \bVec$:
\begin{displaymath}
\dim(P(V)) = f_P(\dim(V)).
\end{displaymath}
We call the polynomial $f_P$ the {\em dimension function} of $P$.
\end{prp}

\begin{dfn}
Let $X$ be a $\bVec$-variety. We define the {\em dimension function}
$f_X$ of $X$ to be:
\begin{displaymath}
f_X(n) \coloneqq \dim(X(K^n)). \qedhere
\end{displaymath}
\end{dfn}

\begin{prp}\label{prp:dimensionfunctionVecvars}
If $X$ is a $\bVec$-variety, then there exists a polynomial $p_X$ with
rational coefficients such that for all $n \gg 0$ we have $f_X(n)=p_X(n)$.
\end{prp}

\begin{proof}
By Noetherianity, we may assume that $X$ is irreducible. By
Proposition~\ref{prp:openisom}, for some nonnegative integer $m$ we have,
for all $n$,
\[ \dim(X(K^m \oplus K^n))=\dim((\oSh_{K^m} X)(K^n)[1/h])=\dim(B) +
\dim(Q(K^n)), \]
so the result follows from Proposition~\ref{prp:dimensionpolyfun}
applied to $Q$.
\end{proof}

There are alternative proofs of this statement, see for example
\cite[Proposition~2.4.7]{danelon:thesis}---indeed, even a more general
version over arbitrary rings with Noetherian spectrum holds, as sketched
in \cite[Proposition~86]{bik-danelon-draisma:topologicalN}.

\comment{
We can refine the above proposition a little, as follows. 

\begin{prp}\label{prp:smallerdegree}
Let $X$ be an closed irreducible $\bVec$-subvariety in a
polynomial functor $P$ of degree $d>0$. Suppose that case (1) of
Theorem~\ref{thm:Embedding} does not apply for any irreducible subfunctor
$R$ of $P$ of degree $d$.  Then the dimension function $f_X$ of $X$
is eventually a polynomial of degree strictly less than $d$.
\end{prp}

\begin{proof}
Since $X \subset P$, we have $f_X(n) \leq p_X(n)$ for all $n$. Hence
$f_X$ is eventually a polynomial of degree $\leq d$. Suppose that it
of degree $d$. Let $U,h,B,Q$ be as in Proposition~\ref{prop:openisom},
and let $\alpha$ be the isomorphism $B \times Q \to \oSh_U(X)[1/h]$.

Then it follows that $\dim(Q(K^n))=\dim(X(U \oplus K^n) - \dim(B)$, so
that $Q$ has degree $d$. Write $Q=Q_{<d} \oplus Q_d$ where $Q_{<d}$
is the sum of the homogeneous components of $Q$ of degree $<d$. By looking at
degrees, we see that the map
\[
B \times Q_{< d} \times Q_d \xrightarrow{\alpha} \oSh_U P = (\oSh_U P)_{<d}
\times P_d
\]
can be decompsed as 
\[ (b,q_{<d},q_d) \mapsto (\alpha_{<d}(b,q_{<d}),\beta_d(b,q_{<d}) +
\rho(b)q_d) \]
where $\rho$ is a morphism from $B$ to the open subscheme in the space
of $\bPF$-homomorphisms $Q_d \to P_d$ consisting of injective maps.

where the latter map is the natural transformation $\oSh_U P \to P$
given by the maps $P(0_{U \to V} \oplus \id_V):P(U \oplus V) \to P(V)$.

By the $\oGL$-equivariance, the transformation $\alpha$ maps $Q_d$ linearly into $P_d$.
Therefore, $X$ contains an open dense subset of the form $Y \times
Q_d$, and with respect to any of the irreducible subobjects of $Q_d$, $X$ satisfies case~\ref{cilindrical} of Theorem~\ref{thm:Embedding}, against our assumption.
\end{proof}
}

\section{Proof of the main theorem}

In this section we prove Theorem~\ref{thm:mainintro}. That is, given a
$\bVec$-variety $X$, we have to show that there exists a unique closed
$\bVec$-subvariety $Y$ of $X$ such that $Y(V)=\oSing(X(V))$ for all $V
\in \bVec$ of sufficiently large dimension; and that this $Y$ satisfies 
$Y(V) \supset \oSing(X(V))$ for all $V \in \bVec$. 

Along the way, following the same inductive approach as
for Theorem~\ref{thm:mainintro}, we will also establish
Theorem~\ref{thm:LocalEqs}.

\subsection{Uniqueness of $Y$}

The uniqueness of the $\bVec$-variety $Y \subset X$ such that $Y(V)
= \oSing(X(V))$ for all $V$ of sufficiently high dimension follows
immediately from the fact that $Y(V)$ for $\dim(V) \gg 0$ determines
$Y(V)$ for small $V$ by Proposition~\ref{prp:Easy}. 

\subsection{Singular points remain singular}

\begin{lm} \label{lm:SingRemainsSing}
Let $X$ be a $\bVec$-variety, $U \in \bVec$, and $p$ a singular point
in $X(U)$. Then for any $V \in \bVec$ and any injective linear map
$\iota:U \to V$, the point $X(\iota)(p)$ is singular in $X(V)$. 
\end{lm}

\begin{proof}
Let $\pi:V \to U$ be a linear map such that $\pi \circ \iota=\id_U$.
Then $X(\pi) \circ X(\iota)=\id_{X(U)}$, and it follows that for any $p
\in X(U)$, the map $T_{X(\iota)(p)} X(\pi):T_{X(\iota)(p)} X(V) \to T_{p}X(U)$ is surjective. Then the claim follows by e.g.\ \cite[Remark 6.8]{chiu-defernex-docampo:embcodimarcs}.
\end{proof}

\begin{lm}
Let $Y$ be a closed $\bVec$-subvariety of $X$ with $Y(V)=\oSing(X(V))$
for all $V \in \bVec$ with $\dim(V) \gg 0$. Then $Y(V)\supset
\oSing(X(V))$
for all $V$.
\end{lm}

\begin{proof}
Let $p \in \oSing(X(V))$ and let $W$ with $\dim(W) \gg 0$, so that
there exists an injective linear map $\iota:V \to W$ and moreover
$\oSing(X(W))=Y(W)$. Then $X(\iota)$ maps $p$ into $\oSing(X(W))=Y(W)$
by Lemma~\ref{lm:SingRemainsSing}. Hence, with $\pi$ as in the proof of
Lemma~\ref{lm:SingRemainsSing}, we have $p=X(\pi)X(\iota) \in Y(V)$.
\end{proof}

In what follows, we will therefore restrict ourselves to the existence
of $Y$ with the property that $Y(V) = \oSing(X(V))$ for all $V$ with
$\dim(V) \gg 0$.

\comment{

Our main result is a far-reaching generalisation of the phenomenon in Example~\ref{xmp:singularlocusboundedmatrices}: we will show that the
singular locus of a $\bVec$-variety is itself a $\bVec$-variety.

\begin{thm}\label{thm:main}
For any polynomial functor $P$ over a field $K$ of
characteristic zero, and for any $\bVec$-variety $X$ of $P$,
there exists a unique $\bVec$-subvariety $Y$ of $X$ such that
for all $V \in \bVec$ of sufficiently large dimension we have
$Y(V)=\oSing(X(V))$.
Furthermore, for {\em all} $V$ we then have
$\oSing(X(V)) \subset Y(V)$.
\end{thm}

We denote the $\bVec$-variety $Y$ from the above theorem by $X^\tsing$, but warn
that, for low-dimensional $V$, it will in general not be true that
$X^\tsing(V)=\oSing(X(V))$; see the following example, which formalises
Example~\ref{xmp:singularlocusboundedmatrices}.

\begin{xmp}\label{xmp:singularboundedrank}
Consider the second tensor power $T^2$ and the $\bVec$-variety $\cM_{\leq r}$ defined by the vanishing of all $(r+1) \times
(r+1)$-subdeterminants.
For $\dim(V) \leq r$, the variety $\cM_{\leq r}(V)$ coincides with the ambient space, with empty singular locus.
On the other hand, for $\dim(V) > r$,
the singular locus of $\cM_{\leq r}(V)$ is precisely $\cM_{\leq r-1}(V)$.
Hence we have $\cM_{\leq r}^\tsing = \cM_{\leq r-1}$.
Note that $\cM_{\leq r}^\tsing(V)=\oSing(\cM_{\leq r}(V))$ for $\dim(V)>r$, while $\oSing(\cM_{\leq r}(V))\subsetneq \cM_{\leq r-1}(V)$ for $\dim(V)\leq r$.
\end{xmp}

}

\subsection{Reduction to $K$ algebraically closed}
\label{ssec:reduction-algebraically-closed}

\begin{lm}
    Let $X$ be a $\bVec$-variety and let $\ol{K}$ denote an algebraic closure of $K$. Write $X_{\ol{K}}$ for the base change of $X$, see Remark~\ref{rmk:extendingbasefield}. If Theorem~\ref{thm:mainintro} holds for the $\bVec_{\ol{K}}$-variety $X_{\ol{K}}$, then it holds for $X$.
\end{lm}

\begin{proof}
    For $V \in \bVec$ write $V_{\ol{K}} \coloneqq V\otimes_K \ol{K}$. Then the morphism
    \[
        \pi_V \colon X_{\ol{K}}(V_{\ol{K}}) = X(V) \times_K \oSpec \ol{K} \to X(V)
    \]
    is closed. Hence, given any closed $\bVec_{\ol{K}}$-subvariety $\widetilde{Y}$ of $X_{\ol{K}}$, setting $Y(V) \coloneqq \pi_V(\widetilde{Y}(V_{\ol{K}}))$ defines a closed $\bVec$-subvariety of $X$.

    Conversely, we have that the image of $\oSing(X_{\ol{K}}(V_{\ol{K}}))$ under $\pi_V$ equals $\oSing(X(V))$. For example, if $X(V)$ is integral, then its singular locus is given by the Fitting ideal $F^d_V \coloneqq\oFitt^d(\Omega_{X(V)/K})$, where $d = \dim X(V)$. But then $X(V) \times_K \oSpec \ol{K}$ is equidimensional of dimension $d$, so its singular locus is given by the pullback of $F^d_V$. Since $\pi_V$ is flat, the image of the closed subset defined by the pullback of $F^d_V$ equals that of $F^d_V$ itself.
\end{proof}

\begin{lm}
Let $X$ be a closed $\bVec$-subvariety of a polynomial functor $P$.
If Theorem~\ref{thm:LocalEqs} holds for $X_{\ol{K}} \subseteq P_{\ol{K}}$,
then it also holds for $X \subseteq P$.
\end{lm}

\begin{proof}
Let $U' \in \bVec_{\ol{K}}$ have the property of $U$ in
Theorem~\ref{thm:LocalEqs}, but for $X_{\ol{K}}$. Without loss of
generality, $U'=U \otimes_K \ol{K}$, where $U \in \bVec_K$. Then the ideal
of $X_{\ol{K}}(U')$ in $\ol{K}[P_{\ol{K}}(U')]$ is obtained from that of
$X(U)$ in $K[P(U)]$ by taking the tensor product with $\ol{K}$. Hence if
$I_V \subseteq K[P(V)]$ for $V \in \bVec_K$ is the ideal of pullbacks as in the theorem
for $X$, and $I_{V \otimes_K \ol{K}}$ is the ideal of pullbacks for
$X_{\ol{K}}$, then $I_{V \otimes_K \ol{K}}=I_V \otimes_K \ol{K}$. By
assumption, the latter defines $X_{\ol{K}}(V \otimes_K \ol{K})$ as a set,
and hence $I_V$ defines $X_V$ as a set. Also by assumption, the reduced
locus of $\oSpec(K[P(V)] \otimes_K \ol{K} / I_V \otimes_K \ol{K})$ contains
the nonsingular locus of $X_{\ol{K}}(V \otimes_K \ol{K})$. Since $K$
has characteristic $0$, the former maps onto the reduced locus of
$\oSpec(K[P(V)]/I_V)$, and the latter maps onto the nonsingular locus of
$X(V)$. This proves Theorem \ref{thm:LocalEqs} for $X \subseteq P$.
\end{proof}

Until the end of this section, we assume that $K = \ol{K}$. By point we mean a closed ($K$-rational) one.

\subsection{Reduction to the irreducible case}

\begin{lm} \label{lm:Irred}
If Theorem~\ref{thm:mainintro} holds for irreducible
$\bVec$-varieties, then it holds for all $\bVec$-varieties.
\end{lm}

\begin{proof}
By Noetherianity (Theorem~\ref{thm:Noetherian}), the $\bVec$-variety $X$ admits
a unique decomposition $X_1 \cup \cdots \cup X_s$ with $s$ a nonnegative integer and each $X_i$ an irreducible $\bVec$-variety that is not
contained in the union $\bigcup_{j \neq i} X_j$.
This means that $X_i(V)
\not\subset \bigcup_{j \in [s]\setminus \{i\}} X_j(V)$ for all vector spaces $V$ of sufficiently large
dimension. Then for each $V\in\bVec$ with $\dim(V) \gg 0$, we have
\[ \oSing(X(V)) = \left(\bigcup_{i = 1}^s \oSing(X_i(V))\right) \cup
\left(\bigcup_{\substack{i,j \in [s];\\i \neq j}} (X_i(V) \cap X_j(V)) \right). \]
By assumption, $X_i$ has a closed $\bVec$-subvariety $Y_i$ such that 
$\oSing(X_i(V)) = Y_i(V)$ for all $V$ with 
$\dim(V) \gg 0$. Now the closed
$\bVec$-subvariety 
\[ Y\coloneqq\left(\bigcup_{i=1}^s Y_i\right) \cup \left(\bigcup_{i \neq j} X_i \cap
X_j\right) \]
has the property that $\oSing(X(V)) = Y(V)$ for all $V$ with
$\dim(V) \gg 0$. 
\end{proof}

\begin{lm}
If Theorem~\ref{thm:LocalEqs} holds for irreducible $\bVec$-varieties,
then it holds for all $\bVec$-varieties.
\end{lm}

\begin{proof}
Let $X=X_1 \cup \cdots \cup X_s$ denote the (non-redundant) decomposition
of $X$ into irreducible $\bVec$-varieties, and let $U \in \bVec$ be
large enough so that it has the property of Theorem~\ref{thm:LocalEqs}
for all $X_i$, as well as the property that no $X_i(U)$ is contained
in any $X_j(U)$ with $i \neq j$. We claim that $U$ then also has the
property of Theorem~\ref{thm:LocalEqs} for $X$.

First of all, if a point $p \in P(V)$ has the property that $P(\phi)(p)
\in X(U)$ for all $\phi \in \oHom(V,U)$, then since $\oHom(V,U)$ is
irreducible, there exists an $i$ such that $P(\phi)(p) \in X_i(U)$ for
all $\phi$, so that $p \in X_i(V)$. Hence the ideal $I_V$ of pullbacks
defines $X(V)$ as a set. 

Next let $V \in \bVec$ and let $p \in X(V)$ be a nonsingular point. We may
assume that $\dim(V)>\dim(U)$, since otherwise the pullbacks of the
equations for $X(U)$ generate a radical ideal in $K[P(V)]$. It then
follows that no $X_i(V)$ is contained in any $X_j(V)$ with $i \neq j$, so
that $X(V)=X_1(V) \cup \cdots \cup X_s(V)$ is the decomposition of $X(V)$
into irreducible components. As $p$ is a nonsingular point of $X(V)$,
it lies in a unique $X_i(V)$, say in $X_1(V)$. 
Let $c = \dim P(V) - \dim_p X_1(V) = \dim P(V) - \dim_p X(V)$. By the irreducible case of Theorem~\ref{thm:LocalEqs}, there exist $f_1, \ldots,f_c \in I(X_1(U))$, $\phi_1,\ldots,\phi_c \in \oHom(V,U)$ and $x_1,\ldots,x_c \in P(V)^*$ such that
\[
    \det (\partial (f_i \circ P(\phi_i))/ \partial x_j)(p) \neq 0.
\]
Note that for fixed $f_1,\ldots,f_c$ this is an open condition for $\phi_1,\ldots,\phi_c$. Therefore we may assume that the latter satisfy in addition the open condition $\phi_i(p) \notin X_j(U)$ for all $j > 1$ and $i = 1, \ldots,c$. Thus there exists $g \in (I(X(U)) : I(X_1(U)))$ such that $(g \circ P(\phi_i))(p) \neq 0$ for all $i = 1,\ldots,c$. Then $gf_1,\ldots,gf_c \in I(X(U))$ satisfy
\[
    \det (\partial (g f_i\circ P(\phi_i))/\partial x_j)(p) \neq 0
\]
and by the Jacobian criterion we have that $I_V = I(X(V))$ locally at $p$.

\end{proof}

\subsection{The irreducible case}
\begin{proof}[Proof of Theorems~\ref{thm:mainintro} for irreducible $X$.]

By definition, the irreducible $\bVec$-variety $X$ 
is a closed $\bVec$-subvariety of a polynomial functor
$P$. We proceed by induction on $P$,
using the well-founded order from Section~\ref{sssec:orderpolyfun}. For $\deg(P)=0$,
$X(V)$ is just a fixed closed subvariety $X(V)=X(0)$, independent of $V$, of the fixed
finite-dimensional affine space $P(V)=P(0)$, and we may set $Y(V)\coloneqq \oSing( X(0))$
for all $V$.

Suppose next that $d\coloneqq \deg(P)>0$, suppose that the theorem holds
for all polynomial functors $Q \prec P$, and let $R$ be an irreducible
subfunctor of the degree-$d$ part $P_d$ of $P$.  We write $P=P' \oplus R$
where $P' \coloneqq P/R \prec P$. Let $\pi:P \to P'$ be the projection
along $R$ and define $X'(V)\coloneqq \overline{\pi_V(X(V))}$ for every
$V \in \bVec$.  Then $X'$ is a closed $\bVec$-variety in $P'$.

By Theorem~\ref{thm:Embedding}, there are two possibilities.  In case (1), 
we have $X(V)=X'(V) \times R(V)$ for all $V$, and then
$\oSing(X(V))=\oSing(X'(V)) \times R(V)$ for all $V$.  By the induction
assumption, there exists a closed subvariey $Y'$ of $X'$ such that
$\oSing(X'(V)) = Y'(V)$ for all $V \in \bVec$
of sufficiently high dimension. Then $Y\coloneqq  Y' \times R$ has the
corresponding property for $X$. 

The previous paragraph applies as long as there exists an irreducible
subfunctor of $P_d$ to which case (1) of Theorem~\ref{thm:Embedding}
applies. So we may assume that such a subfunctor does not exist. In case (2) of Theorem~\ref{thm:Embedding}, we embed a certain open subset of a shift of $X$ into a polynomial
functor smaller than $P$. In what follows, we will have to distinguish
between points where at least one such embedding is defined, and points
where none of them is. We will first argue that the latter points are
automatically contained in the singular locus. 

More precisely, we define a closed $\bVec$-subvariety $Z \subset X$
of $X$ as follows. For $V \in \bVec$, let $I_X(V) \subset K[P(V)]$
be the vanishing ideal of $X$. Then $Z$ is the largest closed
$\bVec$-subvariety of $X$ such that for each $V$ all polynomials in
the set 
\[ \left\{\frac{\partial f}{\partial r} \mid f \in I_X(V),\
R
\text{ an irreducible subfunctor of } P_d, \text{ and } r \in R(V) \right\} \]
vanish identically on $Z(V)$.

\begin{lm}\label{lm:Z}
For all vector spaces $V$ of sufficiently large dimension, $Z(V) \subset
\oSing(X(V))$. 
\end{lm}

\begin{proof}
Since polynomial functors in characteristic zero form a semisimple
category, $P_d$ is the sum of its irreducible subfunctors.
Therefore, the directional derivative $\frac{\partial
f}{\partial r}$ vanishes on $Z(V)$ for all $V \in \bVec$, $f \in
I_X(V)$, and $r \in P_d(V)$. 

We construct the Jacobi-matrix for $X(V)$ as follows: the rows
correspond to a generating set of $I_X(V)$, the first
$n_{<d}=n_{<d}(V)$ columns correspond to coordinates on $P_{<d}(V)$, and
the last $n_d=n_d(V)$ columns correspond to coordinates on $P_d(V)$. By
Proposition~\ref{prp:dimensionpolyfun}, $n_{<d}$ grows as a polynomial
of degree $<d$ in $\dim(V)$ and $n_d$ grows as a polynomial of degree $d$
in $\dim(V)$.

Next, we claim that the codimension $c(V)$ of $X(V)$ in $P(V)$ is a
polynomial of degree $d$ in $V$ for $\dim(V) \gg 0$. Indeed, after
choosing any $R,U,h$ as in case (2) of Theorem~\ref{thm:Embedding}, we
find that $\dim(P(U \oplus V)) - \dim(X(U \oplus V)) \geq \dim(R(V))$,
and the latter is a polynomial of degree $d$ in $V$.

But then, for $\dim(V) \gg 0$, any $c(V) \times c(V)$-submatrix of the Jacobi
matrix intersects the last $n_d$ columns. On $Z(V)$, the last $n_d$
columns are identically zero. Hence $Z(V) \subset \oSing(X(V))$,
as desired.
\end{proof}

By Noetherianity, the closed $\bVec$-subvariety $Z$ of $X$ is
defined by finitely many $h_i\coloneqq  \partial f_i/\partial
r_i$ for $i = 1, \ldots, k$ with $f_i \in I_X(U_i)$ and $r_i \in
R_i(U_i)$, where the $R_i$ are irreducible subfunctors of $P_d$. Now by
Theorem~\ref{thm:Embedding} case (2), $\oSh_{U_i}(X)[1/h_i]$ is isomorphic
to a closed $\bVec$-subvariety $Z_i \subset P''_i$ with $P''_i\coloneqq
K^1 \oplus (\oSh_{U_i}P)/R_i$. Now $P''_i \prec P$, so by the induction
assumption there exists a closed $\bVec$-variety $Y'_i$ of $Z_i $ such
that $Y'_i(V) = \oSing(Z_i(V))$ for all $V$ with 
$\dim(V) \geq m_i$, for some $m_i \in \bbN$.

We use the isomorphism $\oSh_{U_i}(X)[1/h_i](V) \cong Z_i(V)$ for $V
\in \bVec$ to identify $Y'_i(V)$ with a locally closed subset of $X(U_i
\oplus V)$.  By Proposition~\ref{prp:Easy}, $\oGL(U_i \oplus V)$ acts on
the latter variety by automorphisms, hence preserving the singular locus.
Take $\dim(V)  \geq m_i$, so that $Y'_i(V)$ coincides with the locus of
singular points in $X(U_i \oplus V)$ where $h_i$ is nonzero. Then for any
$p \in Y_i'(V)$ the map $\oGL(U_i \oplus V) \to X(U_i \oplus V), g \mapsto g \cdot
p$ maps $\oGL(U_i \oplus V)$ into $\oSing(X(U_i \oplus V))$, and in fact
an open dense subset of $\oGL(U_i \oplus V)$ into $Y'_i(V)$. This implies
that the closure $\overline{Y'_i(V)}$ of $Y'_i(V)$ in $X(U_i \oplus V)$ is
a $\oGL(U_i \oplus V)$-stable closed subset of $\oSing(X(U_i \oplus V))$.

For $W \in \bVec$, we define $Y_i(W) \subset X(W)$ as follows. Choose
any $V \in \bVec$ with $\dim(V) \geq \max\{m_i, \dim(W)-\dim(U_i)\}$
and any surjective linear map $\phi:U_i \oplus V \to W$ and set
\[ Y_i(W)\coloneqq  \overline{P(\phi)Y'_i(V)}. \]

\begin{lm}\label{lm:Yi}
The rule $Y_i$ defines a closed $\bVec$-subvariety of $X$, which moreover
satisfies $Y_i(W) \subset \oSing(X(W))$ for all $W$ with $\dim(W)
\geq m_i+\dim(U_i)$.  More precisely, for such large $W$, $Y_i(W)$
is the closure in $X(W)$ of the set
\[ \{p \in \oSing(X(W)) \mid \exists \psi \in \oHom(W,U_i):
h_i(P(\psi)p) \neq 0\}. \]
\end{lm}

\begin{proof}
To simplify notation, we set $U \coloneqq U_{i}, h \coloneqq h_{i}, 
Y' \coloneqq Y'_{i}$, and $ Y \coloneqq Y_{i}$.
We first show that $Y(W)$ does not depend on the choice of
(sufficiently large) $V$ and of $\phi$. Let $\phi':U \oplus V' \to W$
be another surjective linear map. Without loss of generality, assume
that $\dim(V') \geq \dim(V)$. Write $\ker \phi'=A_1 \oplus A_2$ where
$\dim(A_2)=\dim(V')-\dim(V)$. Then $\phi'$ factors as 
\[ \phi'=\phi \circ g_1 \circ (\oid_U \times \psi) \circ g_2\] 
where $g_2 \in \oGL(U \oplus
V')$ maps $A_2$ into a subspace $B_1$ of $V'$, $\psi:V' \to V$ is a
linear map with kernel $B_1$, and $g_1 \in \oGL(U \oplus V)$ maps the
image of $A_1$ under $(\oid_U \times \psi) \circ g_2$, which has dimension
$\dim(U \oplus V)-\dim(W)$, into $\ker \phi$.  Now, by the discussion above, $g_2$
preserves $\overline{Y'(V')}$, $g_1$ preserves $\overline{Y'(V)}$,
and by definition $P(\oid_U \times \psi)$ maps the dense subset
$Y'(V')$ of the former onto the dense subset $Y'(V)$ of the latter. We
conclude that $\overline{P(\phi')(Y'(V'))} = \overline{P(\phi)(Y'(V))}$, 
as desired. 

By construction, each $Y(W)$ is closed in $X(W)$. 
We next argue that $Y$ is (covariantly) functorial in $W$. Let $\sigma:
W \to W'$ be a linear map. We need to verify that $P(\sigma)$ maps
$Y(W)$ into $Y(W')$. To this end, we observe that we can find a
sufficiently large space $V'$ and linear maps $\sigma':V  \to V'$
and $\phi':U \oplus V' \to W'$ such that the following diagram commutes
and the vertical maps are surjective:
\[ 
\xymatrix{
U \oplus V \ar[r]^{\id_U \oplus \sigma'} \ar[d]_{\phi} & U \oplus V' \ar[d]^{\phi'}\\
W \ar[r]_{\sigma} & W'.
}
\]
Indeed, then we have 
\begin{align*} 
P(\sigma)(Y(W))&=P(\sigma)(\overline{P(\phi)(Y'(V))})
\subset \overline{P(\sigma \circ \phi)(Y'(V))}\\
&=\overline{P(\phi')(P(\id_U \oplus \sigma')(Y'(V)))}
\subset \overline{P(\phi') Y'(V')} 
= Y(W'). 
\end{align*}
So $Y$ is a closed $\bVec$-subvariety of $X$.  Furthermore, the
discussion before Lemma~\ref{lm:Yi} implies that $Y(W)$ is contained in
$\oSing(X(W))$ for $\dim(W) \geq m_i + \dim(U)$. Finally, for $W$ of
this dimension, we may fix any isomorphism $\phi:U \oplus V \to W$,
and then $Y(W)=\overline{P(\phi)Y'(V)}$ by the independence of $Y(W)$
of the choice of $\phi$. The last statement of the lemma now follows
from the fact that $Y'(V)=\oSing(X(U \oplus V)[1/h])$.
\end{proof}

Finally, we claim that 
\[
Y_1(V) \cup \cdots \cup Y_k(V) \cup Z(V) = \oSing(X(V)).
\]
for all $V$ with $\dim(V) \gg 0$.
Indeed, the inclusion $\subset$ follows from Lemma~\ref{lm:Z} and Lemma~\ref{lm:Yi}.
For the other inclusion, pick a point $p\in \oSing(X(V))$.
If $p \in Z(V)$, then we are done. Otherwise, for some $i \in \{ 1,
\dots, k\}$ there exists a $\psi: V \to U_i$ such that $h_i(P(\psi)(p))\neq 0$, and therefore $p \in Y_i(V)$.
In conclusion, $Y\coloneqq  Y_1 \cup \cdots \cup Y_k \cup Z$ is the $\bVec$-subvariety of
$X$ satisfying the conditions of Theorem~\ref{thm:mainintro}. 
\end{proof}

We recall the notation $\Xsing \subset X$ for the unique $\bVec$-subvariety $Y$ from the Theorem~\ref{thm:mainintro}.

\begin{proof}[Proof of Theorem~\ref{thm:LocalEqs} for irreducible
$X$.]
We use the notation, intermediate results, and global induction structure
from the previous proof. First assume that $X=X' \times R$. By the
induction assumption, there exists a $U$
is such that for all $V$, the reduced locus of the scheme defined by
the ideal of $X'(U) \subseteq P'(U)$ pulled back to $P'(V)$ contains
the nonsingular locus of $X'(U)$. Then $U$
has the same property for $X \subseteq P$. 

So we may assume that case (1) of Theorem~\ref{thm:Embedding} does not
apply to any irreducible subfunctor $R$ of the top-degree part $P_d$
of $P$. We then construct $Z,Y_1,\ldots,Y_k \subseteq X$ as in the
proof above, so that for $V$ sufficiently large, $\oSing(X(V))=Y_1(V)
\cup \cdots \cup Y_k(V) \cup Z(V)$. By induction, we may further assume
that Theorem~\ref{thm:LocalEqs} holds for the embedded closed
$\bVec$-subvarieties
$Z_i \subseteq K^1 \oplus (\oSh_{U_i} P)/R_i=:Q_i$. 

Now let $V$ be of sufficiently large dimension and let $p \in X(V)
\setminus \oSing(X(V))$. Then $p$ is not in $Z(V)$, so there exists an $i$
and a $\phi:V \to U_i$ so that $h_i(P(\phi)(p)) \neq 0$. Since this is
an open condition and since $\dim(V)$ is assumed to be sufficiently
large, we may assume
that $\phi$ is surjective, and hence that $V=U_i \oplus W$ with $\phi$
the projection onto $U_i$ with kernel $W$. We can therefore regard $p$
as a point of $(\oSh_{U_i}X)(W)[1/h_i]$. Let $\pi$ be the map $(\oSh_{U_i}
P)[1/h_i] \to Q_i$ whose second component is the projection along $R_i$ and
whose first component is the function $1/h_i$, so $\pi$ restricts to an
isomorphism $(\oSh_{U_i} X)[1/h_i] \to Z_i$. Hence $\pi(p)$ is a nonsingular
point of $Z_i$. 

Let $d = \dim Q_i(W) - \dim_{\pi(p)} Z_i(W)$. By induction, there exist $U \in \bVec$, $\phi_1, \ldots,\phi_d \in \oHom(W,U)$, $g_1,\ldots,g_d \in I(Z_i(U))$ and $x_1,\ldots,x_d \in Q_i(U)^*$ such that
\[
    \det (\partial (g_j \circ Q_i(\phi_j))/\partial x_k)(p) \neq 0.
\]
Pulling back to $(\oSh_{U_i} P)[1/h_i]$ and multiplying by a suitable power of $h_i$ we may assume by abuse of notation that $g_1,\ldots,g_d \in I(\oSh_{U_i}(X)(U))$ with the above property. Write $\widetilde{g}_j \coloneqq g_j \circ P(\id_{U_i} \oplus \phi_j)$ for $j = 1, \ldots, d$.

Finally, let $x_{d+1},\ldots,x_c$ be a basis of $R_i(W)^*$, and recall
that $h_i = \frac{\partial f_i}{\partial r_i}$ for some $f_i$ in the
ideal of $X(U_i)$ and some $r_i \in R(U_i)$. The proof of the Embedding
Theorem in \cite{draisma} shows that the linear span of the pull-backs of $f_i \circ P(\psi)$,
where $\psi$ runs over all linear maps $U_i \oplus W \to U_i$ that are
the identity on $U_i$, contains polynomials of the form
\[ 
    \widetilde{g}_j \coloneqq x_j h_i + r_j \text{ for all } j=d+1,\ldots,c, 
\]
where $r_j \in K[P(U_i \oplus W)/R(W)]$ does not involve the variables
$x_{d+1},\ldots,x_c$. Then $\widetilde{g}_1,\ldots,\widetilde{g}_c$ are elements of $I_V$ such that $\det (\partial \widetilde{g}_j/\partial x_k)(p) \neq 0$. Moreover, we have $c = d + \dim R_i(W) = \dim P(V) - \dim_p X(V)$, so by the Jacobian criterion we are done.

\end{proof}

\section{\texorpdfstring{$\bVec$}{Vec}-varieties of linear type}\label{sec:lineartype}

Recall from Proposition~\ref{prp:dimensionfunctionVecvars} that for a $\bVec$-variety $X$, $\dim(X(K^n))$ is a polynomial
in $n$ for all sufficiently large $n$. In the remaining sections, we
study those $\bVec$-varietes $X$ for which this polynomial is linear.

\begin{dfn}
The $\bVec$-variety $X$ is said to be {\em of linear type} if there
exist $d,b \in \bQ$ such that $\dim(X(K^n))=dn+b$ for all $n \gg 0$.
\end{dfn}

For $\bVec$-varieties $X$ of linear type we will find an explicit number
$n$ such that $\Xsing(V)=\oSing(X(V))$ for all $V$ with $\dim(V) \geq
n$ (in this section) and describe a weak resolution of singularities
(in the next section). For both of these purposes, it will be convenient to 
switch freely between the following different characterisations of linear-type
$\bVec$-varieties.

\begin{prp}\label{prp: equivalent conditions linear type}
Let $X$ be an irreducible $\bVec$-variety and let $d \in \bbN$.
The following conditions are equivalent:
    \begin{enumerate}

	\item $X$ admits a dominant morphism $B \times (S^1)^d \to X$
	for some finite-dimensional variety $B$. 

	\item For every vector space $V \in \bVec$ we have
	$X(V)=\oHom(K^d,V) \cdot X(K^d)$, where the right-hand side is short-hand notation
	for the  image of the morphism $\oHom(K^d,V)
	\times X(K^d) \to X(V),\ (\phi,p) \mapsto X(\phi)(p)$.

	\item $X$ is of linear type and the coefficient of $n$ in the linear polynomial 
	$n \mapsto \dim(X(K^n))$ (for $n \gg 0$) is $\leq d$.
    \end{enumerate}
\end{prp}

\begin{proof}
$(1) \implies (2)$: Let $\alpha:B \times (S^1)^d \to X$ be a dominant
morphism. Then for any point $p \in X(V)$ in the image of
$\alpha_V$, there exists a field extension $L$ of $K$ such that we can write $p=\alpha(b,v_1,\ldots,v_d)$ where $v_i \in V \otimes_K L$
and $b$ is a $L$-point of $B$. Let $\phi$
be the $L$-point of $\oHom(K^d,V)$ sending $e_i$ to $v_i$. Then it
follows that $p=X_L(\phi)(\alpha_{L^d}(b,e_1,\ldots,e_d))$. Since $\alpha$ is dominant,
we find that $\oHom(K^d,V) \cdot X(K^d)$ is dense in $X(V)$.

For $\dim(V)<d$ we already know that $\oHom(K^d,V) \cdot X(K^d)$ is all of $X(V)$---indeed,
$X(\phi):X(K^d) \to X(V)$ is surjective for any surjective linear map
$\phi:K^d \to V$. Now assume that $\dim(V) \geq d$. Then, since any
linear map $K^d \to V$ factors as $\iota \circ \psi$ for some
$\psi:K^d \to K^d$ and some injective $\iota:K^d \to V$, and since $X(K^d)$ is
mapped into itself by $X(\psi)$, we have
\[ \oHom(K^d,V) \cdot X(K^d)=\oHom^0(K^d,V) \cdot X(K^d), \]
where $\oHom^0(K^d,V) \subset \oHom(K^d,V)$ is the open subscheme of {\em
injective} maps. Now the right-hand side is also the locus of points $p$
where the space $U_p$ from Lemma~\ref{lm:UpSemicont} has dimension $\leq
d$, hence closed by semicontinuity. Together with the conclusion of the previous paragraph,
this implies that $\oHom(K^d,V) \cdot X(K^d)=X(V)$.

$(2) \implies (3)$: For any $n$ we have a surjective
morphism $\oHom(K^d,K^n) \times X(K^d) \to X(K^n)$. The left-hand side
has dimension $dn + \dim(X(K^d))$, and hence $X$ is of linear type and the coefficient of $n$ in the 
dimension polynomial of $X$ is at most $d$.

$(3) \implies (1)$: By Proposition~\ref{prp:openisom}, we have $(\oSh_U
X)[1/h] \cong B \times P$ for some pure polynomial functor $P$ and some
irreducible affine variety $B$. Since the dimension polynomial of $X$ is
linear, so is that of $P$, and they have the same coefficient of $n$, so
that $P=(S^1)^e$ for some $e \leq d$. We have a surjective morphism
$\oSh_U
X \to X$, which, evaluated at $V \in \bVec$, is the map $X(U \oplus V)
\to X(U)$ corresponding to the projection $U \oplus V \to V$ with kernel
$U$. Pre-composing this with the projection $B \times (S^1)^d \to B \times
(S^1)^e$ and the isomorphism $B \times (S^1)^e \to (\oSh_U
X)[1/h]$ gives the desired dominant morphism $B \times (S^1)^d \to X$.
\end{proof}

As a consequence of Proposition~\ref{prp: equivalent conditions linear type}, 
the minimal number $d \in \bbN$ for which each of the conditions
holds is the same for all.

\begin{rmk}
    If $L\supset K$ is a field extension, then clearly $X$ is of linear type if and only if $X_L$ is of linear type; and in this case the minimal numbers $d$ from Proposition~\ref{prp: equivalent conditions linear type} agree, as the dimension polynomials of $X$ and $X_L$ agree.
\end{rmk}

\begin{xmp}
Many varieties of tensors studied in applications are of linear type. For
instance, the variety of order-$m$ tensors of border tensor rank at most
$r$ is the image closure of the morphism
\[ (S^1)^{mr} \to T^m, (u_{ij})_{i \leq r, j \leq d} \mapsto \sum_{i=1}^r
u_{i1} \otimes \cdots \otimes u_{im}, \]
which by Proposition~\ref{prp: equivalent conditions linear type}
is of linear type; the minimal $d$ equals $mr$ here.
An example of a variety of tensors not of linear type is the variety of tensors of order $d\geq 3$ of bounded slice rank; see Example~\ref{xmp:slicerank}.
\end{xmp}

Let $X$ be an irreducible $\bVec$-variety $X$ of linear type
and let $d$ be as in Proposition~\ref{prp: equivalent conditions
linear type}. Assume that $X$ is a closed $\bVec$-subvariety of the
polynomial functor $Q$. Then the assignment that sends $V \in \bVec$
to the linear span of $X(V)$ in $Q(V)$ is a subfunctor of $Q$. Assume that it is all of $Q$, i.e., that {\em $X$ spans $Q$}.
    
\begin{lm} \label{lm:lengthatmostd}
If the linear-type $\bVec$-variety $X$ spans $Q$, then 
every partition $\lambda$ for which the Schur functor
    $\bbS_\lambda$ appears in $Q$ has length $\leq d$. 
\end{lm}

\begin{proof}
    For any $V \in \bVec$ and any point $p$ of $X(V)$, by item (2)
    in Proposition~\ref{prp: equivalent conditions linear type} and Lemma~\ref{lm:minimaldimension},
    there exists a field extension $L$ of $K$ and subspace $U \subset V \otimes_K L$ of dimension at most
    $d$ such that $p \in X_L(U) \subset Q_L(U) \subset Q_L(V \otimes_K L)$. For
    all $\lambda$ of length $>d$ we have $\bbS_\lambda(U)=0$ by
    \cite[Theorem~6.3]{fulton-harris:reprfirstcourse}, so since $X$
    spans $Q$, these do not appear in $Q$.
\end{proof}

\begin{thm}\label{prp:F(d,c)}
    Let $X$ be an irreducible $\bVec$-variety of linear type that admits
    a dominant morphism from $B \times (S^1)^d$, where $B$ is an irreducible
    affine variety of dimension $c$. 
    Then there exists a number $F(d, c)$, depending only on $d$ and on $c$, 
    such that for every vector space $V$ of dimension $d + k$ with $k  \geq F(d,c)$, we have that $X^\tsing(V) = \oSing(X(V))$.
\end{thm}

Following the argument in
Section~\ref{ssec:reduction-algebraically-closed}, it suffices to
prove the statement for $X_{\ol{K}}$, where $\ol{K}$ denotes an algebraic closure of $K$. Thus, for the proof of Theorem~\ref{prp:F(d,c)}, we may assume that $K$ is algebraically closed and that any point is $K$-rational.

\begin{proof}
    We may take $d$ minimal, so that by Proposition~\ref{prp: equivalent
    conditions linear type} we have $f_{X}(n)=d \cdot n + c'$ for all
    $n \gg 0$, where $c'$ is a constant $\leq c$. Moreover, we will see
    in Lemma~\ref{lm:dimlineartype}, that this formula holds for all $n
    \geq d$.

    Assume that $X$ is a closed $\bVec$-subvariety of the polynomial
    functor $P$. Without loss of generality, we may further assume that
    $X$ spans $P$. To any point $p$ of $X(K^d)$, we associate to $p$ a
    functor $T_p: \bVec \to \bVec$ as follows. First, we regard $p$ as a
    point in $X(K^d \oplus V) \subset P(K^d \oplus V)$ for every $V \in
    \bVec$ via the natural inclusion $P(K^d) \to P(K^d \oplus V)$. Then
    $T_p$ sends $V$ to $T_p(X(K^d\oplus V))$ inside $T_p(P(K^d
    \oplus V)) = P(K^d \oplus V)$.  The functor $T_p$ is a polynomial
    functor since the tangent space is a vector space, and for every
    $\phi: V \to W$ the differential of the map $P_{\oid_{K^d} \oplus \phi}$
    is $P_{\oid_{K^d} \oplus \phi}$ itself and it maps $T_p(X(K^d\oplus V))$
    to $T_p(X(K^d\oplus W))$.

    Let $f_{T_p}$ and $f_X$ be the dimension functions of $T_p$ and
    $X$, respectively. Now by Lemma~\ref{lm:SingRemainsSing}, either $p$
    is a nonsingular point in $X(K^d \oplus V)$ for all $V$, in which case
    $f_X(d+k)=f_{T_p}(k)$ for all $k$, or else $p$ is a singular point
    in $X(K^d \oplus V)$ for all $V$ of sufficiently large dimension,
    and then $f_X(d+k)<f_{T_p}(k)$ for $k \gg 0$. The following claim expresses
    that this dichotomy can be tested at a value of $k$ that depends only on
    $c$ and $d$.
    
    \begin{clm}\label{clm:inner}
        There exists a number $F(d,c)$ depending only on $d$ and on $c$ such that either for all $k \geq  F(d,c)$ we have $f_{T_p}(k)> f_X(d+k)$ or for all $k \geq F(d,c)$ we have $f_{T_p}(k) = f_X(d+k)$.
    \end{clm}
    
    \begin{proof}[Proof of Claim~\ref{clm:inner}]
    Note that $T_p$ is a subfunctor of $\oSh_{K^d}P$.
    By Lemma~\ref{lm:lengthatmostd}, $P$ does not contain any Schur functors
    $\bbS_\lambda$ for which $\lambda$ has length strictly larger than $d$. 
    Hence, by Remark~\ref{rmk:boundedlength}, neither does $\oSh_{K^d}P$.
    In particular, the dimension function $f_{T_p}$ is a sum of dimension
    functions $f_{\bbS_\mu}$ for partitions $\mu$ whose length $d'$ is at most
    $d$. Using the dimension formula for $\dim(\bbS_{\mu}(V))$ from 
    \cite[Theorem~6.3]{fulton-harris:reprfirstcourse}, one verifies readily that 
    $f_{\bbS_{\mu}}(k) \geq f_{\Wedge^{d'}}(k) = \binom{k}{d'}$ for all $k$.

    If $T_p$ contains a Schur functor $\bbS_\mu$
    with $\mu$ of length $d'>1$, then $f_{T_p}(k) \geq f_{\Lambda^{d'}}(k)
    = \binom{k}{d'}$ holds for every $k$. Let $n_{d'}$ be the minimal integer
    such that for every $k \geq n_{d'}$ we have $\binom{k}{d'} > d(d+k) + c$.

    If $T_p$ contains no $\bbS_\mu$ with $\mu$ of length $>1$ but
    $T_p$ has degree strictly bigger than one, then $T_p$ is 
    a sum of symmetric powers, and $f_{T_p}(k) \geq f_{S^2} (k)=\binom{k+1}{2}$ for all $k$. 
    Let $n_0$ be the minimal integer such that for every $k  \geq n_0$
    we have $\binom{k+1}{2} > d(d+k) + c$.
    
    If $T_p$ has degree one, then $f_{T_p}(k) = d'' k + c''$ for some
    non-negative constant $c''$.  If $d''>d$, then $f_{T_p}(k) \geq
    (d+1)k$. Let $n_1$ be such that for all $k  \geq n_1$ we have $(d+1)k
    > d(d+k) + c$.

    Now set $F(d,c) \coloneqq \max \{n_0, n_1, n_2, \dots, n_d\}$. By
    construction, in any of the above cases, we have $f_{T_p}(k) > d(d+k)
    + c \geq \dim(X(K^d \oplus V))$ for any space $V$ of dimension $k
    \geq F(d,c)$.

    If $T_p$ has degree one and $d''=d$, then since $\dim(X(K^d \oplus
    V))= d(d+\dim(V))+c'$ for all $V$, the difference $f_{T_p}(k)-f_{X}(d+k)$
    is a constant nonnegative function for all $k \geq 0$, hence either
    identically $0$ or identically $>0$.
    \end{proof}

    To conclude the proof of Theorem~\ref{prp:F(d,c)}, let $k$ be
    at least the number $F(d,c)$ from the claim, and let $V \in \bVec$ be  of dimension $n = d + k$. We argue that $X^\tsing
    (V) = \oSing(X(V))$. The inclusion $\supset$ always holds by
    Theorem~\ref{thm:mainintro}. For the converse, consider a point
    $p$ of $X^\tsing(V)$. Then, by Proposition~\ref{prp: equivalent
    conditions linear type} there is a point $p' \in X(K^d) \subset
    X(K^d \oplus V')$, for any vector space $V'$ of dimension $\dim(V)-d$, and an isomorphism $\phi: K^d\oplus V' \to V$
    such that $P(\phi)(p') = p$. Then it follows that $p' \in X^\tsing(K^d)$,
    so that the dimension of $T_{p'}(W)(=T_{p'}(X(K^d \oplus W)))$ is strictly greater than
    $\dim(X(K^d \oplus W))$ for all $W$ with $\dim(W) \gg 0$. 
    The property of $F(d,c)$ then implies that this inequality already holds for
    $W=V'$, so that $f_{T_{p'}}(k) > f_X(d+k)$.
    We conclude that $p' \in \oSing(X(K^d \oplus V'))$ and hence $p \in
    \oSing(X(V))$.
\end{proof}

The following example illustrates some of the reasoning and notation above. 

\begin{xmp} \label{xmp:rankr}
    Let $T^2$ be the polynomial functor given by taking the second tensor power, namely, for $V \in \bVec$ we have $T^2(V) \coloneqq V \otimes V$.
    Let $\cM_{\leq r}\subset T^2$ be the $\bVec$-subvariety of matrices
    of rank $\leq r$. This admits a dominant (indeed, surjective) morphism
    \[ (S^1)^{2r} \to \cM_{\leq r},\ (u_1,v_1,\ldots,u_r,v_r) \mapsto \sum_{i=1}^r u_i
    \otimes v_i. \]
    Since also $\dim(\cM_{\leq r}(K^n))=n^2 - (n-r)^2= 2rn - r^2$, where the coefficient
    of $n$ is $2r$, $d\coloneqq2r$ is the minimal exponent of $S^1$ that we can
    choose in such a dominant morphism. Furthermore, $\oSh_{K^d}T^2 = T^2(K^d) \oplus (S^1)^{2d} \oplus \bigwedge^2 \oplus S^2$.
    Now pick a point $p \in \cM_{\leq r}(K^{d})$, i.e., a $d \times d$-matrix,
    and regard $p$ as a point in $\cM_{\leq r}(K^d \oplus V)$ for any $V$. If $p$
    has rank precisely $r$, then the minimal spaces $U_1,U_2 \subset
    K^d$ with $p \in U_1 \otimes U_2$ satisfy $\dim(U_1)=\dim(U_2)=r$,
    and we find that
    \[ T_p(\cM_{\leq r}(K^d \oplus V))=U_1 \otimes (K^{d} \oplus V) + (K^{d} \oplus V)
    \otimes U_2
    \]
    where the intersection of the two spaces on the
    right-hand side equals $U_1 \otimes U_2$. We then find that
    $f_{T_p}(k)=2r(d+k)-r^d=f_{\cM_{\leq r}}(d+k)$. On the other hand, if $p$ has
    rank $<r$, then adding any rank-$1$ matrix to $p$ we stay within $X$, and
    since the space of matrices is spanned by rank-$1$ matrices, 
    $T_p$ is all of $\oSh_{K^d} T^2$.
\end{xmp}

\comment{
\begin{rmk}
    Equivalently, Theorem~\ref{prp:F(d,c)} says that there exists a function $F(d, c)$ with the following property.
    Let $V$ be a vector space of dimension $d + k$ such that $k > F(d, c)$.
    Then, if $p \in X(V)$ satisfies $P(\iota)(p) \in \oSing(X(V \oplus V'))$ for some $V'\in \bVec$ via the inclusion $\iota: V \to V \oplus V'$, then $p \in \oSing(X(V))$.
\end{rmk}
https://arxiv.org/pdf/1204.0488
}

\section{Weak resolution}\label{sec:weakresolution}

Given a linear-type $\bVec$-variety $X$, we want to assign to each $V\in \bVec$ a smooth reduced scheme $\Omega(V)$ and a morphism $\pi_V: \Omega(V) \to X(V)$ such that $\pi_V$ is a $\oGL(V)$-equivariant resolution of singularities.
Moreover, we would like that the assignment $\Omega$ behaves as close to a functor as possible.
More specifically we want $\Omega$ to be behave well with  respect to injective linear maps.
Our set-up doesn't allow this much, but something going in this direction. 
First of all, we start with looking at the following example.
\begin{xmp}\label{xmp:resolutionmatrices}
    Building on Example~\ref{xmp:singularlocusboundedmatrices} and
    Example~\ref{xmp:rankr}, we note that for every $V \in \bVec$ there is a $\oGL(V)$-equivariant resolution of singularities.
    Consider the $\bVec$-variety $\cM_{\leq r}$ that maps each $V\in \bVec$ in the subspace of $T^2(V) = V\otimes V$ of matrices of rank at most $r$.
    Then for every $V \in \bVec$ the space
    \[
    \Omega(V) \coloneqq \{(p, U, W) \mid p \in U \otimes W\} \subset V\otimes V \times \oGr(r, V)\times \oGr(r, V)
    \]
    together with the projection $\pi_V$ on the first component, is a resolution of singularity of $\cM_{\leq r}(V)$.
    Indeed, $\Omega(V)$ is smooth as it is a vector bundle on $\oGr(r, V) \times \oGr(r, V)$.
    It is proper because the Grassmannians are complete.
    Finally, there is an obvious inverse defined on the smooth locus: a rank $r$ matrix uniquely defines $U$ and $W$ as, respectively, the row-span and the column-span.
    The assignment $\Omega$ is also semi-functorial: every injective linear map $\phi: V \to V'$ defines a map $\Omega(\phi): \Omega(V) \to \Omega(V')$ sending the triple $(p, U , W)$ to $(T^2(\phi)(p), \phi(U), \phi(W))$, and additionally the assignment $\pi$ makes the following diagram commute:
    \begin{center}
        \begin{tikzcd}
            \Omega(V) \arrow[r, "\Omega(\phi)"] \arrow[d, "\pi_V"] & \Omega(V') \arrow[d, "\pi_{V'}"]\\
            X(V) \arrow[r, "X(\phi)"] & X(V'). 
        \end{tikzcd}
    \end{center}
\end{xmp}

\subsection{Preliminaries on torsors} \label{ssec:Torsors}

Let $Y,\Gamma$ be varieties over $K$ and let $G$ be an algebraic group over
$K$ that acts from the left on $Y$. For us, a morphism $\pi:Y \to \Gamma$
is called a {\em $G$-torsor} if  $\pi$ is constant on $G$-orbits and
moreover $\Gamma$ can be covered by $G$-stable Zariski-open subsets $U$
for which there exists a $G$-equivariant isomorphism $\iota:\pi^{-1}(U) \to G
\times U$, where $G$ acts by left multiplication on the left factor,
such that the following diagram:
\[ \xymatrix{
\pi^{-1}(U) \ar[r]^\iota \ar[d]_{\pi} & G \times U  \ar[dl]\\
 U, &  }
\]
where the diagonal map is the projection, commutes. (Often, local
triviality is required to hold only in the \'etale topology, but for
us, the above more restrictive notion of torsor suffices.) Note that $Y$
is smooth if and only if $\Gamma$ is smooth, and that $\Gamma$ is a categorical
quotient of $Y$ by $G$: every morphism $\phi$ from $Y$ that is
constant on $G$-orbits factors as $\psi \circ \pi$ for a unique morphism
$\psi$. Explicitly, $\psi$ can be defined on $U$ above as
\[ \psi(u)\coloneqq\phi(\iota^{-1}(e,u)); \]
it is straightforward to check that that this is well-defined. We
may therefore write $\Gamma=Y/G$.

Next assume that $G$ also acts from the left on a third variety $Z$. Then
we have a diagonal left action of $G$ on $Y \times Z$, and we claim
that there is a $G$-torsor $\widetilde{\pi}$ from $Y \times Z$ to a suitable
variety $\Omega\eqqcolon(Y \times Z)/G$. To construct $\Omega$,
for every pair $U \subset \Gamma$ and $\iota:\pi^{-1}(U) \to G \times U$
as above, we take a copy of $U \times Z$. When $U',\iota':\pi^{-1}(U')
\to G \times U'$ is another such pair, then set $V\coloneqq U \cap U'$.  The map
$\iota' \circ \iota^{-1}:G \times V \to G \times V$ is $G$-equivariant
and commutes with projection to $V$, hence it is of the form $(g,u)
\mapsto (g  h(u),u)$ for some morphism $h:V \to G$.
We can think of this map as
identifying the open subset $G \times V$ in the chart $G \times U$ of $Y$ 
with the open subset $G \times V$ in the chart $G \times U'$ of $Y$.

We now glue $V \times Z \subset U \times Z$ to $V \times Z \subset
U' \times Z$ via $(u,z) \mapsto (u,h(u)^{-1}z)$. This glueing yields
$\Omega$. The morphism $\widetilde{\pi}: Z \times Y \to \Omega$ is defined
on the chart $(G \times U) \times Z $ of $Y \times Z$ via the map $(g,u,z)
\mapsto (u,g^{-1}z)$ to the chart $U \times Z$ of $\Omega$. This is,
indeed, constant on $G$-orbits, since $g'(g,u,z)=(g'g,u,g'z)$ is
mapped to the same point as $(g,u,z)$.  And $\widetilde{\pi}$ is also
well-defined: for $u \in V=U \cap U'$ the point $(g,u,z)$ corresponds to
the point $(g h(u),u,z)$ in the chart $(G \times U') \times Z$, which
is mapped to $(u,h(u)^{-1} g^{-1} z)$ in the chart $U' \times Z$, and
this corresponds to the point $(u,g^{-1} z)$ in the chart $U \times Z$,
as desired. It is straightforward to verify that $\widetilde{\pi}$ is,
indeed, a $G$-torsor.

Finally, we observe that the map $Y \times Z \to \Gamma,\ (y,z) \mapsto
\pi(y)$ is constant on $G$-orbits, and hence factors via a morphism
$\Omega \to \Gamma$.

\subsection{Construction of $\Omega$} \label{ssec:Omega}
For the remainder of this section let $X$ be an irreducible $\bVec$-variety of linear type and
let $d$ be the minimal number satisfying the properties in Proposition~\ref{prp:
equivalent conditions linear type}. 
Let $\rho: Z \to X(K^d)$ be a $\oGL_d$-equivariant resolution of singularities.
Such a resolution exists by \cite[Corollary~7.6.3]{villamayor}, and \cite[Proposition~7.6.2]{villamayor} guarantees that the action of $\oGL_d$ on $Z$ induced by the lifting of the action on $X(K^d)$ is still an algebraic group action.

Let $V\in \bVec$ be a vector space of dimension at least $d$. In the notation
of Section~\ref{ssec:Torsors}, choose
$Y=Y(V)\coloneqq\oHom^0(K^d, V)$, the open subscheme of $\oHom(K^d,V)$ of
injective linear maps, and $\Gamma=\Gamma(V)\coloneqq\oGr(d,V)$. The morphism $Y
\to \Gamma$ that maps the linear map $\phi \in \oHom^0(K^d,V)$ to
its image $\im(\phi)$, a point in $\oGr(d,V)$, is a well-known $\oGL_d$-torsor with
respect to the action $(g,\phi) \mapsto \phi \circ g^{-1}$.

Consequently, Section\ref{ssec:Torsors} yields a  new torsor
\[ \oHom^0(K^d,V) \times Z \to (\oHom^0(K^d,V) \times Z)/\oGL_d = \Omega(V),\
(\phi,z) \mapsto [\phi,z]. \]
By Section\ref{ssec:Torsors}, the space $\Omega(V)$ comes with a morphism
$\Omega(V) \to \oGr(d,V)$. But in our present context we have a second
morphism from $\Omega(V)$. Indeed, the morphism
\[ \oHom^0(K^d,V) \times Z \to X(V),\ (\phi,z) \mapsto X(\phi)(\rho(z)) \]
is invariant on $\oGL_d$-orbits, because
\[ X(\phi \circ g^{-1})(\rho(gz))=X(\phi \circ g^{-1})(g \rho(z))=X(\phi \circ
g \circ g^{-1}(\rho(z))=X(\phi)(\rho(z)). \]
Here the first equality uses that $\rho: Z \to X$ is $\oGL_d$-equivariant
and the second equality uses that $X$ is a functor.
Note that by Remark~\ref{rmk:image-of-hom} this description holds for all $R$-points with $R$ a $K$-algebra.
Since $\Omega(V)$
is a categorical quotient of $\oHom^0(K^d,V) \times Z$ by $\oGL_d$,
this implies that the morphism factors via a morphism $\rho_V: \Omega(V)
\to X(V)$. We summarise our findings in the following commuting diagram:
\begin{equation} \label{eq:Omega}
\xymatrix{ 
& \oHom^0(K^d,V) \times Z \ar@/_1.5pc/[ddl]_{(\phi,z) \mapsto \im(\phi)}
\ar[d]^{/\oGL_d}
\ar@/^1.5pc/[ddr]^{(\phi,z) \mapsto X(\phi)\rho(z)}&   \\
& \Omega(V) \ar[dl] \ar[dr]_{\rho_V} & \\
\oGr(d,V) & & X(V).
}
\end{equation}
We observe that $\Omega(V)$ is smooth because $\oHom^0(K^d,V) \times
Z$ is smooth and the vertical map is a $\oGL_d$-torsor. Furthermore,
$\Omega(V)$ is {\em semi-functorial} in $V$ in the following sense. Let
$\psi: V \to W$ be any linear map of rank at least $d$. Then for $\phi$
in an open dense subscheme of $\oHom^0(K^d,V)$ we have $\psi \circ \phi
\in \oHom^0(K^d,V)$. We obtain we the following commuting diagram, where
the dashed maps are rational maps:
\begin{equation} \label{eq:OmegaFunc}
\xymatrix
{
\oHom^0(K^d, V) \times Z \ar@{-->}[rr]^{(\phi,z) \mapsto (\psi \circ \phi,z)} 
\ar[d] 
&&
\oHom^0(K^d,W) \times Z \ar[d] \\
\Omega(V) \ar@{-->}[rr]^{[\phi,z] \mapsto [\psi \circ \phi,z]} \ar[d]_{\rho_V} 
&& \Omega(W) \ar[d]^{\rho_W} \\
X(V) \ar[rr]_{X(\psi)} && X(W).
}
\end{equation}

\begin{xmp}
    \label{xmp:incidence-variety}
    Let us revisit the subspace variety in Definition~\ref{dfn:incidence-variety}. That is, for a polynomial functor $P$ and $d\geq0$ we have a $\bVec$-variety $\Lambda_{d,P}$ defined by
    \[
        \Lambda_{d,P}(V) \coloneqq \{p \in P(V) \mid \dim(U_p) \leq d\},
    \]
    where $U_p$ is as in Lemma~\ref{lm:minimaldimension}. Clearly $\Lambda_{d,P}$ is of linear type, with $d$ being the minimal number such that the assertions in Proposition~\ref{prp: equivalent conditions linear type} are fulfilled. Moreover, we have $\Lambda_{d,P}(K^d) = P(K^d)$. So in the above construction of $\Omega$ we may choose $Z = P(K^d)$, and it is straightforward to check that $\Omega(V) \simeq \Gamma_{d,P}(V)$, the $d$-th incidence variety of $P$ (again, see Definition~\ref{dfn:incidence-variety}).
\end{xmp}

\subsection{Weak resolution}
Our goal is to show that the maps $\rho_V: \Omega(V) \to X(V)$,
constructed in Section~\ref{ssec:Omega} and depending semifunctorially on
$V$, are weak resolutions of the varieties $X(V)$ with $\dim(V) \geq
d$. Explicitly, we will prove the following.

\begin{thm} \label{thm:WeakResolution}
Let $X$ be an irreducible $\bVec$-variety of linear type such that the
coefficient  of $n$ in the polynomial $\dim(X(K^n))$ (for $n \gg 0$) is equal to $d \in
\bbN$. For every $V \in \bVec$ of dimension at least $d$, the variety
$\Omega(V)$ constructed in Section~\ref{ssec:Omega} is smooth, and the map
$\rho_V:\Omega(V) \to X(V)$ is proper and birational.
\end{thm}

In particular, it will follow from the proof that $\dim(X(K^n))$ equals
a polynomial in $n$ for all $n \geq d$.

Smoothness of $\Omega(V)$ was already established in Section~\ref{ssec:Omega}, so it
remains to show that $\rho_V$ is proper and birational. 

\comment{
\begin{prp}\label{prp:resolutionmap}
    For every $V \in \bVec$ of dimension $\geq d$ the map $\rho_V: \Omega(V) \to X(V)$
    is surjective.
\end{prp}

\begin{proof}
For $q$ be a point in $X(V)$. By Proposition~\ref{prp: equivalent
conditions linear type}, there exists a point $\phi$ in $\oHom^0(K^d,V)$
and a point $p \in X(K^d)$ such that $X(\phi)(p)=q$. Since $\rho:Z
\to X(K^d)$ is a resolution of singularities, there exists a $z \in Z$
with $\rho(z)=p$.  Now $\rho_V([\phi,z])=q$ by definition of $\rho_V$.
\end{proof}
}

\begin{prp} 
    For every $V \in \bVec$ of dimension $\geq d$ the map $\rho_V:
    \Omega(V) \to X(V)$ is proper.
\end{prp}

It is convenient to keep diagrams~\eqref{eq:OmegaFunc} and,
especially,~\eqref{eq:Omega} at hand in the following argument.

\begin{proof}
    By \cite[\href{https://stacks.math.columbia.edu/tag/0208}{Tag 0208}]{stacks-project} it suffices to prove the valuative criterion for DVRs. That is, assume $R$ is a DVR with quotient field $L$ and we have a diagram
    \begin{equation}
        \label{eq:valuative-proper}
        \begin{tikzcd}
            \oSpec L \arrow[r, "\omega"] \arrow[d] & \Omega(V) \arrow[d, "\rho_V"] \\
            \oSpec R \arrow[r] & X(V).
        \end{tikzcd}
    \end{equation}
    Then we need to show the existence of a diagonal arrow $\oSpec R \to \Omega(V)$; such an arrow is unique as $\rho_V$ is separated. We lift $\omega$ to $(\varphi,z) \colon \oSpec L \to \oHom^0(K^d,V) \times Z$ and identify $\varphi$ with a $n\times d$-matrix with coefficents in $L$. Using the Smith normal form (over PIDs) we find $g\in \oGL_d(L)$ such that $\varphi \circ g^{-1}$ has all coefficients in $R$. After acting with $g$ we may therefore assume that $\omega$ lifts to $(\varphi,z)$ with $\varphi$ an $R$-point of $\oHom^0(K^d,V)$.

    The morphism $\rho_V$ is induced by the composition
    \[
        \begin{tikzcd}
            \oHom^0(K^d,V) \times Z \arrow[r, "(\oid{,}\rho)"] & \oHom^0(K^d,V) \times X(K^d) \arrow[r] & X(V).
        \end{tikzcd}
    \]
    The base change of the second morphism with respect to $\varphi$ yields a map $\widetilde{\varphi} \colon \oSpec R \times X(K^d) \to \oSpec R \times X(V)$. Now choose an $R$-point $\sigma$ of $\oHom(V,K^d)$ such that $\sigma \circ \varphi$ is the identity matrix. As before, the base change of $\oHom(V,K^d) \times X(V) \to X(K^d)$ with respect to $\sigma$ yields a morphism $\widetilde{\sigma} \colon \oSpec R \times X(V) \to \oSpec R \times X(K^d)$ such that $\widetilde{\sigma} \circ \widetilde{\varphi} = \oid$. Thus the composition
    \[
        \begin{tikzcd}
            \oSpec R \times Z \arrow[r, "(\oid{,}\rho)"] & \oSpec R \times X(K^d) \arrow[r, "\widetilde{\varphi}"] & \oSpec R \times X(V) \arrow[r, "\widetilde{\sigma}"] & \oSpec R \times X(K^d)
        \end{tikzcd}
    \]
    equals $(\oid,\rho)$. Then the diagram~\eqref{eq:valuative-proper} gives rise to
    \[
        \begin{tikzcd}
            \oSpec L \arrow[r] \arrow[d] & \oSpec R \times Z \arrow[d,"(\oid{,}\rho)"] \\
            \oSpec R \arrow[r] & \oSpec R \times X(K^d)
        \end{tikzcd}
    \]
    and there exists a diagonal arrow $\widetilde{z}$ since $\rho$ is proper. The point $\widetilde{\omega} \coloneqq (\varphi, \widetilde{z})$ induces a diagonal arrow in the original diagram~\eqref{eq:valuative-proper}. 
\end{proof}

\begin{prp} \label{prp:Birational}
    For every $V \in \bVec$ of dimension $\geq d$ the map $\rho_V:
    \Omega(V) \to X(V)$ is birational. Moreover, $\rho_V$ is an isomorphism over the set $S(V)$ given as the intersection of $\Lambda_{d,X}(V)\setminus \Lambda_{d-1,X}(V)$ and the nonsingular locus of $X(V)$.
\end{prp}

\begin{proof}
Recall that for all $e\geq 0$ the subspace variety $\Lambda_{e,X}(V)$ from Definition~\ref{dfn:incidence-variety} is a closed subvariety of $X(V)$.  By Proposition~\ref{prp: equivalent conditions
linear type} we have that $S(V)$ is dense open. We claim it is sufficent to prove that $\rho \colon S'(V) \coloneqq \rho_V^{-1}(X^0(V)) \to X^0(V)$ is a bijection on $\ol{K}$-rational points. Indeed, then the base change $\rho_{\ol{K}} \colon S'(V)_{\ol{K}} \to S(V)_{\ol{K}}$ of $\rho$ is a bijection on closed points. By Zariski's Main theorem, $\rho_{\ol{K}}$ factors as the composition of an open immersion and a finite morphism $\rho'$. But since $\ol{K}$ is of characteristic $0$, the degree of $\rho'$ equals the degree of a fiber over a general closed point, which is $1$. As $S(V)_{\ol{K}}$ is smooth and in particular normal, we have that $\rho'$ and thus $\rho_{\ol{K}}$ is an isomorphism. But then $\rho$ itself is isomorphism.

We now prove the claim. Let $p$ be a $\ol{K}$-rational point $X^0(V)$. By Lemma~\ref{lm:SingRemainsSing} the image of $p$ in $X(U_p)$ is nonsingular.
Let $[\phi,z] \in \Omega(V)$ be such that $\rho_V([\phi,z])=p$. This
means that $X(\phi)(\rho(z))=p$, and hence $p \in X(\im(\phi))$. But then
$\im(\phi)=U_p$ by Lemma~\ref{lm:minimaldimension}, so $p$ determines the
$\oGL_d$-orbit of $\phi$. Suppose that also $p=\rho_V([\phi \circ g,z'])$
for some $\ol{K}$-points $g$ of $\oGL_d$ and $z'$ of $Z$. The right-hand
side equals $\rho_V([\phi,gz'])$, and since $X(\phi)$ is injective,
we have $\rho(gz')=\rho(z)$. Since $X(\phi)$ is an isomorphism from
$X(K^d)$ to $X(U_p)$ and since $p$ is a nonsingular point of $X(U_p)$,
$\rho(z)$ is a nonsingular point of $X(K^d)$. Since $\rho$ is a resolution
of singularities, we have $gz'=z$. Hence the fibre over $p$ consists of
a single point.
\end{proof}

\begin{rmk}
    \label{rmk:Birational-inverse}
    Let us give an explicit local description of the inverse of $\rho_V$ for $V = K^n$. Let $p_0$ be a point in $S(V)$. In what follows we suppress residue fields (resp.\ their algebraic closures) in our notation. Assume, without loss of generality, that the projection of $U_{p_0}$ onto the first $d$ coordinates is surjective. Then for $p$ in a Zariski-open neighbourhood $B \subset S(V)$ of $p_0$, there is a unique linear map $\phi_p$ whose matrix is of the form
    \[ 
        \begin{bmatrix} I_d \\ * \end{bmatrix} 
    \]
    and which satisfies $\im(\phi_p)=U_p$. The map $p \mapsto \phi_p$ is a morphism. Let $\sigma \in \oHom(K^n, K^d)$ be the projection to the first coordinates, so that $\sigma \circ \phi_p$ is the $d \times d$-identity matrix and $\phi_p \circ \sigma$ restricts to the identity on $U_p$. Then, for all $p\in B$, $X(\sigma)$ maps $X(U_p)$ isomorphically to $X(K^d)$, and hence maps $p \in X(U_p)$ to a smooth point of $X(K^d)$, to which, since $\rho$ is a resolution of singularities, we can apply the inverse of $\rho$. Now
    \[ 
        p \mapsto [\phi_p, \rho^{-1}(X(\sigma)(p))] 
    \]
    is a morphism, and applying $\rho_V$ to the right-hand side yields
    \[
        X(\phi_p)(X(\sigma)(p))=p, 
    \]
    as desired.
\end{rmk}

\begin{cor}\label{lm:dimlineartype}
For every $V \in \bVec$ of dimension $\geq d$ we have 
\[ \dim(X(V)) = dim(X(K^d)) + d(\dim(V) - d). \]
\end{cor}

\begin{proof}
By Proposition~\ref{prp:Birational} we have $\dim(X(V))=\dim(\Omega(V))$
and by the construction we have
\begin{align*} &\dim(\Omega(V))=\dim(\oHom^0(K^d,V)) + \dim(Z) - \dim(\oGL_d)\\
&= d \dim(V) + \dim(X(K^d)) - d^2 = \dim(X(K^d)) + d(\dim(V)-d).  \qedhere
\end{align*}
\end{proof}

This concludes the proof of Theorem~\ref{thm:WeakResolution}.

\subsection{Examples}

We now illustrate some strengths and weaknesses of the weak
resolution $\rho_V:\Omega(V) \to X(V)$ through a number of examples. We
start with an example where $\rho_V$ agrees with an earlier example.

\begin{xmp}[Example~\ref{xmp:resolutionsymmetric} revisited]
   Let $X \subset S^2$ be the variety of symmetric matrices of rank
   $\leq r$.  Then $d=r$, since any $p \in X(V)$ of rank equal to $r$
   lies in $X(U)$ for unique $r$-dimensional $U \subset V$. We have
   $X(K^r)=S^2(K^r)$, which is smooth, and hence $Z=X(K^r)$ and $\rho$
   is the identity. Now, for $V$ of dimension $\geq r$, we have
   \[ \Omega(V)=(\oHom^0(K^r, V) \times S^2(K^r))/\oGL_r \]
   and $\rho_V([\phi,p])=S^2(\phi)p$. This is readily seen to
   be isomorphic, even functorially in $V$, to the resolution of
   Example~\ref{xmp:resolutionsymmetric}.
\end{xmp}

Next an example that shows that $\rho_V$ is, in general, not an actual
resolution of singularities. 

\begin{xmp}
    Consider the $\bVec$-variety $X = S^1$. We have
    $\dim(X(K^n))=\dim(K^n)=n$, so $d=1$. The variety $X(K^1)$ is the
    affine line $\bbA^1$, hence smooth, so $Z=X(K^1)=\bbA^1$, and $\rho$ is
    the identity map. Now, for all $V$ of dimension at least $1$, we have
    \[ \Omega(V)=(\oHom^0(K^1,V) \times \bbA^1)/\oGL_1  \cong 
    ((V \setminus \{0\}) \times \bbA^1)/\oGL_1,  \]
    and under this isomorphism we have $\rho_V([v,z])=zv$. So $\rho_V$
    is the blow-up of the affine space $V$ in the origin.
\end{xmp}

\begin{xmp}[Example~\ref{xmp:resolutionmatrices} revisited]
Let $\cM_{\leq r} \subset T^2$ be the variety of rank $\leq
r$-matrices. The smallest subspace $U$ for which a rank-$r$ element of $V
\otimes V$ lies in $U \otimes U$ is typically $2r$-dimensional, so $d=2r$. 
As in Example~\ref{xmp:resolutionmatrices} we define 
\[ Z\coloneqq\{(p,U,W) \in (K^{2r} \otimes K^{2r}) \times \oGr(r,K^{2r}) \times
\oGr(r,K^{2r}) \mid p \in U \otimes W\} \]
and let $\rho:Z \to X(K^{2r})$ be the projection to the first factor. We
have
\[ \Omega(V)=(\oHom^0(K^{2r},V) \times Z)/\oGL_{2r} \]
and $\rho_V([\phi,(p,U,W)] = T^2(\phi)(p)$. We find that the map $\rho_V$
has positive-dimensional fibres over all matrices $p \in X(V)$ whose
row space and column space {\em together} span a space of dimension
$<2r$. This includes all singular points of $X(V)$, whose row and column
space each are of dimension $<r$, but many more. On the other hand,
$\rho_V$ is an isomorphism over the open subscheme ov $X(V)$ where the
row and column space together span a space of dimension equal to $2r$.
\end{xmp}

\begin{rmk}
In the setting of Proposition~\ref{prp:Birational}, a point $p \in
X(V)$ for which $U_p$ has dimension $d$ is nonsingular in $X(V)$ if
and only if it is nonsingular in $X(U_p)$. ``Only if'' is the content
of Lemma~\ref{lm:SingRemainsSing} (and was used above). For ``if'',
we observe that the morphism defined in Remark~\ref{rmk:Birational-inverse} is actually
defined on all points for which $X(U_p)$ is nonsingular. 

Recall that in Section\ref{sec:lineartype} we established the existence of a
bound $F(d,c)$ such that for all $V \in \bVec$ of dimension $d+k$ with
$k \geq F(d,c)$ we have $X^\tsing(V)=\oSing(X(V))$. It would be nice to
have a much more explicit bound, e.g. linear in $d$. The main obstacle
to obtaining such a bound is that for nonsingular points $p \in X(K^d)$
with $\dim(U_p)<d$ we do not yet have a more explicit bound on $k$ where
we need to test whether $p$ is still nonsingular in $X(K^d \oplus K^k)$.
\end{rmk}

\printbibliography

\end{document}